\documentclass[12pt]{amsart}

\usepackage[margin=1.15in]{geometry}
\usepackage{amscd,amssymb, amsmath,amsfonts, wasysym, mathrsfs, enumerate, mathtools,hhline,color}
\usepackage{graphicx}
\usepackage[all, cmtip]{xy}

\usepackage{url}

\definecolor{hot}{RGB}{65,105,225}

\usepackage[pagebackref=true,colorlinks=true, linkcolor=hot ,  citecolor=hot, urlcolor=hot]{hyperref}

\theoremstyle{plain}
\newtheorem{theorem}{Theorem}[section]
\newtheorem{proposition}[theorem]{Proposition}

\newtheorem{corollary}[theorem]{Corollary}

\newtheorem{lemma}[theorem]{Lemma}

\theoremstyle{definition}
\newtheorem{definition}[theorem]{Definition}
\newtheorem{question}[theorem]{Question}
\newtheorem{remark}[theorem]{Remark}

\theoremstyle{assumption}

\newtheorem*{ex*}{Example}

\newtheorem{property}[theorem]{Property}
\numberwithin{equation}{section}

\theoremstyle{bigthm}
\newtheorem{bigthm}{Theorem}

\DeclareMathOperator{\id}{id}                    

\title[Self-Covering]{Self-Covering, Finiteness, And Fibering Over A Circle}

\begin{document}

\author{Lizhen Qin}
\address{Department of Mathematics, Nanjing University, 22 Hankou Road, Nanjing, Jiangsu 210093, China}
\email{qinlz@nju.edu.cn}

\author{Yang Su}
\address{HLM, Academy of Mathematics and Systems Science, Chinese Academy of Sciences, Beijing 100190, China}
\address{School of Mathematical Sciences, University of Chinese Academy of Sciences, Beijing 100049, China}
\email{suyang@math.ac.cn}

\author{Botong Wang}
\address{Department of Mathematics, University of Wisconsin-Madison, 480 Lincoln Drive, Madison, WI 53706, USA}
\email{bwang274@wisc.edu}

\maketitle

\begin{center}
Dedicated to Francis Thomas Farrell on the occasion of his 80th birthday.
\end{center}

\begin{abstract}
A topological space is called self-covering if it is a nontrivial cover of itself. We prove that a closed self-covering manifold $M$ with free abelian fundamental group fibers over a circle under mild assumptions. In particular, we give a complete answer to the question whether a self-covering manifold with fundamental group $\mathbb Z$ is a fiber bundle over $S^1$, except for the $4$-dimensional smooth case. As an algebraic Hilfssatz, we  develop a criterion for finite generation of modules over a commutative Noetherian ring.  We also construct examples of self-covering manifolds with non-free abelian fundamental group, which are not fiber bundles over $S^1$.
\end{abstract}

\section{Introduction}\label{sec_introduction}
A topological space is called self-covering if it is homeomorphic, or more generally homotopy equivalent, to a nontrivial covering space of itself. Self-covering spaces, especially self-covering manifolds, are well-studied in the literature (see e.g. \cite{BDT}, \cite{Dere}, \cite{VanLimbeek} and \cite{Wang_Wu}). Many works about self-covering spaces including this paper are motivated by the following questions.
\begin{quote}
\emph{Which closed manifolds admit a self-covering structure? And what geometric structures does a self-covering manifold possess?}
\end{quote}
The simplest self-covering closed manifold is a circle. In this paper, we study when a closed self-covering manifold is a fiber bundle over a circle.

We work in the category of smooth (DIFF), piecewise-linear (PL) and topological (TOP) manifolds simultaneously. Throughout this paper, the abbreviation CAT stands for any of these categories. In particular, a CAT manifold is a smooth, piecewise-linear or topological manifold, and a CAT isomorphism is a diffeomorphism, a piecewise-linear homeomorphism or a homeomorphism, etc.

We start with a simple observation. Let $F$ be a closed CAT manifold and $f$ be an automorphism of $F$. Then the quotient space $M=F \times [0,1] / \sim$ with $(x,0) \sim (f(x), 1)$ is a CAT fiber bundle over $S^{1}$ with fiber $F$ and monodromy $f$. For each integer $k>0$, let $M_k=F \times [0,1]/\sim$ with $(x,0) \sim (f^{k}(x), 1)$. Then $M_{k}$ is a $k$-fold cover of $M$. If there is an automorphism $g \colon F \to F$ such that $f^{k} g$ is CAT pseudo-isotopic to $gf$ or $g f^{-1}$, then $M$ is isomorphic to $M_{k}$ and we obtain a self-covering manifold $M$, which is a fiber bundle over $S^1$. The motivation of this paper is to study whether every self-covering manifold can be constructed in this way. In other words, we study the following question.

\begin{question}\label{que_general}
Suppose $M$ is a connected closed CAT $n$-manifold, which is isomorphic to a $k$-fold cover $M_{k}$ of itself for some $k>1$. Is $M$ a CAT fiber bundle over $S^{1}$ with monodromy $f$ such that $f^{k} g$ is CAT pseudo-isotopic to $gf$ or $g f^{-1}$ for some automorphism $g$ of $F$?
\end{question}

The answer to Question \ref{que_general} is negative in general as we construct two types of counterexamples in Theorems \ref{thm_no_finite} and \ref{thm_no_fiber}. However, under restrictions of the fundamental group, one may obtain a positive answer. A special case of Question \ref{que_general} is the following.

\begin{question}\label{que_pi1_z}
Is a connected closed self-covering CAT $n$-manifold $M$ with $\pi_{1} (M) = \mathbb{Z}$ a CAT fiber bundle over $S^{1}$?
\end{question}

Though depending on both the dimension and the category, the answer to Question \ref{que_pi1_z} is almost always positive. The question is trivial for $n \le 2$. For $n=3$, an affirmative answer is given by the classical theory of $3$-manifolds.  We will give an answer to Question \ref{que_pi1_z} in dimension $\ge 4$ in this paper. The main result is a fibering theorem for self-covering manifolds with free abelian fundamental group in dimension $\ge 5$ (Theorem \ref{thm_high_manifold}).

To fiber a manifold $M$ over a circle, one could apply the classical fibering theorem of Browder-Levine \cite{Browder_Levine}, or more generally, the theorems of Farrell \cite{Farrell} and Siebenmann \cite{Siebenmann}. We would like to emphasize that all these theorems rely on the assumption that the infinite cyclic cover $M_{\infty}$ of $M$ induced by certain epimorphism $\pi_{1} (M) \rightarrow \mathbb{Z}$ is homotopy equivalent to a finite CW complex. This finiteness condition is indeed necessary for $M$ being fibered over a circle. In fact, for a fiber bundle $M$ over $S^{1}$,  the infinite cyclic cover corresponding to the epimorphism $\pi_1(M) \to \pi_1(S^1)$ is homotopy equivalent to the fiber $F$, which is a closed manifold.  Therefore our key step toward the fibering theorem of self-covering manifolds is to deduce the homotopy finiteness from the self-covering assumption. This is a problem in homotopy theory. So we first study the infinite cyclic coverings of CW complexes, which itself has a long history and many applications  (see e.g. \cite{Cowsik_Swarup}, \cite{Hillman_Kochloukova}, \cite{Milnor68} and \cite{Shinohara_Sumners}). In this paper, we prove that certain finiteness and periodicity (Theorems \ref{thm_cw_homotopy}, \ref{thm_cw} and \ref{thm_cw_period}) of the infinite cyclic covers can be derived from the self-covering assumption of the bases.


\medskip

To begin with, we  state our finiteness theorem for CW complexes. Let $X$ be a connected CW complex with fundamental group $\pi_{1} (X) = G \times \mathbb{Z}$, where $G$ is a finitely generated abelian group.  Let $X_{k}$ be the finite cyclic cover of $X$ with $\pi_{1} (X_{k}) = G \times (k \mathbb{Z})$, and let $X_{\infty}$ be the infinite cyclic cover with $\pi_{1} (X_{\infty}) = G \times 0 $. Recall that a CW complex is of finite type if there are only finitely many cells in each dimension, whereas the dimension of the complex may be infinite; a space is finitely dominated if it is dominated by a finite CW complex.

\begin{bigthm}\label{thm_cw_homotopy}
Let $X$ be defined as above. Suppose there exists a homotopy equivalence $h: X \rightarrow X_{k}$ with $k>1$ and $h_{\sharp} (G \times 0) = G \times 0$, where $h_{\sharp}: \pi_{1} (X) \rightarrow \pi_{1} (X_{k})$ is the isomorphism induced by $h$. If $X$ is homotopy equivalent to a CW complex of finite type (resp. is finitely dominated), then so is $X_{\infty}$.
\end{bigthm}

\begin{remark}
The assumption $h_{\sharp} (G\times 0) = G \times 0$ in Theorem \ref{thm_cw_homotopy} is reflected in the following geometric picture. Suppose $X$ is a fiber bundle over $S^{1}$ with fiber $F$ and $\pi_{1} (F) = G \times 0$. Then $X_{k}$ is also a bundle over $S^{1}$ with the same fiber. If the map $h: X \rightarrow X_{k}$ preserves the bundle structure, then $h_{\sharp} (G \times 0) = G \times 0$.
\end{remark}

\begin{remark}\label{rem_finite}
In Theorem \ref{thm_cw_homotopy}, even if $X$ is homotopy equivalent to a finite CW complex, the conclusion can not be strengthened to ``$X_{\infty}$ is homotopy equivalent to a finite CW complex". For example, if $Y$ is finitely dominated but not homotopy equivalent to a finite CW complex, then $X = Y \times S^{1}$ is homotopy equivalent to a finite CW complex with $X_{\infty} \simeq Y$.  In Theorem \ref{thm_no_finite}, we present a closed smooth manifold as an additional counterexample.
\end{remark}

Now we consider self-covering manifolds under the same assumption. Namely, let $M$ be a connected closed CAT manifold with fundamental group $\pi_1(M) = G \times \mathbb Z$, where $G$ is an abelian group. Let $M_{k}$ be the finite cyclic cover of $M$ with $\pi_{1} (M_{k}) = G \times (k \mathbb{Z})$. Assume that there exists a homotopy equivalence $h: M \rightarrow M_{k}$ for some $k>1$ such that the induced isomorphism $h_{\sharp} \colon \pi_1(M) \to \pi_1(M_k)$ satisfies $h_{\sharp}(G \times 0)=G \times 0$. In this case, the composition
\begin{equation}\label{eqn_deck}
\mathbb{Z} \rightarrow G \times \mathbb{Z} \overset{h_{\sharp}}{\longrightarrow} G \times (k\mathbb{Z}) \rightarrow k\mathbb{Z}.
\end{equation}
is an isomorphism, and hence the image of $1 \in \mathbb Z$ is either $k$ or $-k$.

\begin{definition}\label{def_deck}
If the composition of (\ref{eqn_deck}) maps 1 to $k$, we say that $h$ \emph{preserves the orientation of the deck transformations}. Otherwise, we say that $h$ \emph{reverses the orientation of the deck transformations}.
\end{definition}

\begin{bigthm}\label{thm_high_manifold}
Let $M$, $G$ and $h: M \rightarrow M_{k}$ be defined as above. Suppose that $G = \mathbb{Z}^{r}$ for some $r \geq 0$, and $\dim M \geq 5$. When $\dim M =5$, we further assume that $G=0$ and the category is TOP. Then the following statements hold.
\begin{enumerate}
\item The manifold $M$ is a CAT fiber bundle over $S^{1}$ with connected fiber $F$ and $\pi_{1} (F) = G \times 0$.

\item There is a homotopy equivalence $g: F \rightarrow F$ such that $f^{k} g$ is homotopic to $gf$ (resp. $g f^{-1}$) when $h$ preserves (resp. reverses) the orientation of the deck transformations, where $f: F\to F$ is the monodromy action.

\item  Furthermore, if $h$ is a CAT isomorphism, then there is a CAT automorphism $g \colon F \to F$ such that $f^{k} g$ is CAT pseudo-isotopic to $gf$ (resp. $g f^{-1}$) when $h$ preserves (resp. reverses) the orientation of the deck transformations.
\end{enumerate}
\end{bigthm}

Theorem \ref{thm_high_manifold} gives a positive answer to Question \ref{que_general} in dimension $> 5$ under the assumption that $\pi_{1} (M) = \mathbb{Z}^{r} \times \mathbb{Z}$ and $h_{\sharp} (\mathbb{Z}^{r} \times 0) = \mathbb{Z}^{r} \times 0$, and in dimension $5$ as well for the TOP category when $\pi_{1} (M) = \mathbb{Z}$.

\begin{remark}
When $\dim M =5$, even if $G=0$, the conclusion of Theorem \ref{thm_high_manifold} is no longer true for DIFF and PL, see Theorem \ref{thm_5_diff}.
\end{remark}

Furthermore, the monodromy $f$ in Theorem \ref{thm_high_manifold} has certain homological periodicity (Corollary \ref{cor_period}, which is a consequence of Theorem \ref{thm_cw_period}), and for special classes of manifolds, $f$ is indeed periodic up to isotopy (Corollary \ref{cor_5_top} and Corollary \ref{cor_monodromy}).

\begin{corollary}\label{cor_period}
Let $M$, $h \colon M \rightarrow M_{k}$, and $f \colon F \rightarrow F$ be as in Theorem \ref{thm_high_manifold}, where $h$ is a homotopy equivalence. Then there exist positive integers $m$ prime to $k$, and $l$ such that
\[
f_{*}^{m} =\id: H_{\bullet} (F; \mathbb{Z}) /T \rightarrow H_{\bullet} (F; \mathbb{Z}) /T \quad \text{and} \quad f_{*}^{l} = \id: H_{\bullet} (F; \mathbb{Z}) \rightarrow H_{\bullet} (F; \mathbb{Z}),
\]
where  $H_{\bullet} (F; \mathbb{Z}) = \oplus_{j=0}^{\infty} H_j(F;\mathbb Z)$ and $T$ is the torsion subgroup of $H_{\bullet} (F; \mathbb{Z})$.
\end{corollary}

\begin{corollary}\label{cor_5_top}
Let $M$, $h$, $F$ and  $f$ be as in Theorem \ref{thm_high_manifold} with $\dim M =5$ (hence CAT = TOP and $G=0$), where $h$ is a homotopy equivalence.  Then $f^m$ is topologically isotopic to $\mathrm{id}_F$ for some $m>0$ prime to $k$.
\end{corollary}

Recall that a smooth manifold is almost parallelizable if the tangent bundle is trivial outside an arbitrary point.
\begin{corollary}\label{cor_monodromy}
Let $M$, $h$, $F$ and  $f$ be as in Theorem \ref{thm_high_manifold}. Assume $M$ is a smooth $(2n+1)$-manifold with $n \geq 3$, $\pi_{1} (M) = \mathbb{Z}$ and $\pi_{i} (M) =0$ for all $1<i<n$. Assume $h$ is a diffeomorphism and $F$ is almost parallelizable, then $f^l$ is smooth isotopic to $\mathrm{id}_F$ for some $l >0$.
\end{corollary}

\begin{remark}
In Corollary \ref{cor_monodromy}, to fulfill the assumption that $F$ is almost parallelizable, it suffices to assume instead either $M$ is almost parallelizable or $n=3$. Actually, if $n=3$, then $F$ is $2$-connected and $\dim F= 6$, and hence $F$ is a connected sum of copies of $S^{3} \times S^{3}$ (c.f. \cite[Theorem~1]{Wall66}).
\end{remark}

Note that, if $f^{l}$ is isotopic to $\mathrm{id}_{F}$ for some $l>0$, then $M_{l}$ is a trivial bundle, and $M_{l+q}$ is isomorphic to $M_{q}$ as bundles for all $q \geq 1$, i.e., the cyclic covering spaces $\{ M_i \ | \ i >0\}$ of $M$ are periodic.

\medskip

The conclusion of Theorem  \ref{thm_high_manifold} may fail if we drop the assumption that $G = \mathbb{Z}^{r}$. There exist self-covering manifolds $M$ such that $M_{\infty}$ are not homotopy equivalent to finite CW complexes (Theorem \ref{thm_no_finite}). There also exist self-covering manifolds $M$ which are not fiber bundles over a circle, though $M_{\infty}$ are homotopy equivalent to  finite CW complexes (Theorem \ref{thm_no_fiber}). In the statement of the following theorems, we use the standard notations in  number theory  (cf.~\cite[Theorem 4.10]{Washington}): for a prime number $p$, $h_{p}^{-}$ (resp. $h_{p}^{+}$) is the first (resp. second) factor of the class number of $\mathbb{Z} [\xi_{p}]$, where $\xi_{p} = \exp (2 \pi \sqrt{-1}/p)$.

\begin{bigthm}\label{thm_no_finite}
Let $p$ be a prime number such that both $h_{p}^{-}$ and $h_{p}^{+}$ have odd prime factors. Then for each integer $n \geq 9$, there exists a connected closed smooth $n$-manifold $M$ such that:
\begin{enumerate}
\item $M$ is stably parallelizable and $\pi_{1} (M) = \mathbb{Z}/p \times \mathbb{Z}$;

\item $M$ is diffeomorphic to a finite cover $M_{k}$ with $\pi_{1} (M_{k}) = \mathbb{Z}/p \times (k \mathbb{Z})$ for some $k>1$;

\item The infinite cyclic cover $M_{\infty}$ with $\pi_1(M_{\infty})= \mathbb Z/p \times 0$ is not homotopy equivalent to any finite CW complex.
\end{enumerate}
\end{bigthm}

\begin{bigthm}\label{thm_no_fiber}
Let $p$ be the one in Theorem \ref{thm_no_finite}. For each $n \geq 5$ and each connected closed smooth $(n-1)$-manifold $F$ with $\pi_{1} (F) = \mathbb{Z}/p$, there exists a connected closed smooth $n$-manifold $M$ such that:
\begin{enumerate}
\item $M$ is smooth $h$-cobordant to $F \times S^{1}$, hence $\pi_{1} (M) = \mathbb{Z}/p \times \mathbb{Z}$ and $M_{\infty}$ is homotopy equivalent to $F$;

\item $M$ is diffeomorphic to $M_{k}$ for all $k>1$;

\item For all $q \geq 1$, $M_{q}$ is not a topological fiber bundle over $S^{1}$.
\end{enumerate}
\end{bigthm}

Prime numbers satisfying the conditions in Theorem \ref{thm_no_finite} and \ref{thm_no_fiber} can be found by comparing the table in \cite[p.~352]{Washington} and the main table in \cite[p.~935]{Schoof}. An instance is $p=191$.

In the above examples, since $\pi_{1} (M) = \mathbb{Z}/p \times \mathbb{Z}$, an isomorphism $h: M \rightarrow M_{k}$ automatically satisfies the condition $h_{\sharp} (G \times 0) = G \times 0$ with $G = \mathbb{Z}/p$. Therefore they satisfy all the assumptions in Theorem \ref{thm_high_manifold} except that $G = \mathbb{Z}^{r}$. These examples also show that the answer to Question \ref{que_general} is negative in general.

\medskip

Now we focus on the special case of Question \ref{que_general} --- Question \ref{que_pi1_z}. As we have seen, by Theorem \ref{thm_high_manifold}, it  has an affirmative answer in dimension $>5$ and in dimension $5$ for the TOP category. We list the the answer to Question \ref{que_pi1_z} in low dimensions as follows.

In dimension $3$, the categories DIFF, PL and TOP are the same. It's well-known that a closed $3$-manifold $M$ with $\pi_{1} (M) = \mathbb{Z}$ is a fiber bundle over $S^{1}$ (without the self-covering assumption). In fact, by the solution of Poincar\'{e} Conjecture, such a manifold is prime \cite[p.\ 27]{Hempel}, and hence is homeomorphic to the product $S^1 \times S^2$ if it is orientable, and homeomorphic to the twisted product $S^{1} \widetilde{\times} S^{2}$ if it is non-orientable (\cite[Theorem~5.2]{Hempel}).

In dimension $4$, we have an analogous result for self-covering manifolds with fundamental group $\mathbb Z$.
\begin{bigthm}\label{thm_4_top}
Let $M$ be a connected closed topological $4$-manifold with $\pi_{1} (M) = \mathbb{Z}$. If $M$ is homotopy equivalent to $M_{k}$ for some $k>1$, then $M$ is homeomorphic to the product $S^{1} \times S^{3}$ if it is orientable, or homeomorphic to the twisted product $S^{1} \widetilde{\times} S^{3}$ if it is non-orientable. Here $S^{1} \widetilde{\times} S^{3}$ is the $S^{3}$-bundle over $S^{1}$ with monodromy $f(x_{1}, x_{2}, x_{3}, x_{4}) = (x_{1}, x_{2}, x_{3}, -x_{4})$.
\end{bigthm}

By Theorem \ref{thm_4_top}, a smooth self-covering $4$-manifold with fundamental group $\mathbb Z$ is homeomorphic to $S^{1} \times S^{3}$ or $S^{1} \widetilde{\times} S^{3}$. If it is a smooth fiber bundle over $S^{1}$, the fiber must be $S^3$ by the solution of Poincar\'e Conjecture. The space of orientation-preserving diffeomorphisms of $S^3$ is path-connected (Cerf's Theorem \cite[p.~3, Theorem 1]{Cerf1}, see also \cite{Hatcher}), therefore a smooth $S^3$-bundle over $S^1$ is diffeomorphic to $S^{1} \times S^{3}$ or $S^{1} \widetilde{\times} S^{3}$. Hence in this case, a positive answer to Question \ref{que_pi1_z} is equivalent to the uniqueness of smooth structures of $S^1 \times S^3$ and $S^{1} \widetilde{\times} S^{3}$.

In dimension $5$ the answer to Question \ref{que_pi1_z} is negative in the smooth category.

\begin{bigthm}\label{thm_5_diff}
There is a connected closed smooth $5$-manifold $M$ with $\pi_{1} (M) = \mathbb{Z}$ satisfying the following properties. The manifold $M$ is diffeomorphic to $M_{k}$ for some $k>1$; however, $M_{q}$ is not a smooth (piecewise linear) fiber bundle over $S^{1}$ for all $q \geq 1$.
\end{bigthm}

We summarize the answer to Question \ref{que_pi1_z} in the table

\medskip

\begin{center}
\begin{tabular}{c|c|c|c|c}
& $\dim \le 3$ & $\dim=4$ & $\dim=5$ & $\dim \ge 6$ \\
\hline
DIFF & yes & unknown yet& no & yes \\
PL & yes & unknown yet& no & yes \\
TOP & yes & yes & yes & yes
\end{tabular}
\end{center}

\medskip

To show the finiteness properties of the infinite cyclic cover, we are led to questions in commutative algebra. In fact, we need Theorem \ref{thm_noether} as the algebraic foundation of this paper, which may be of independent interest. This theorem helps us to reduce a problem of modules to a problem of vector spaces.

Before stating the theorem, we set the notations and recall some basic facts from commutative algebra. See Section \ref{sec_algebra} for more details. Let $R$ be a Noetherian commutative ring and $S$ be a commutative $R$-algebra with generators $t_{1}, \dots, t_{n}$. Here by an $R$-algebra, we mean $S$ is equipped with a ring homomorphism $\iota: R \rightarrow S$. In particular, $S$ is also a Noetherian commutative ring, and any $S$-module is naturally an $R$-module. Let $\mathrm{Spec} (R)$ be the prime spectrum of $R$, i.e., the set of all prime ideals of $R$. For each $P \in \mathrm{Spec} (R)$, let $\kappa (P)$ be the residue field of $R$ at $P$, i.e., the quotient field of $R/P$.

Let $\mathfrak{L}$ be an $S$-module and $P \in \mathrm{Spec} (R)$, we consider the following property for a pair $(\mathfrak{L},P)$.

\begin{property}\label{pt_eigenvalue}
As a $\kappa (P)$-vector space, $\kappa (P) \otimes_{R} \mathfrak{L}$ is finite-dimensional and the eigenvalues of the induced action of $t_{i}$ on $\kappa (P) \otimes_{R} \mathfrak{L}$ are integral over $R/P$ for every $1 \leq i \leq n$.
\end{property}

Let $\mathfrak{M}$ be a nonzero finitely generated  $S$-module. By definition, a prime ideal $P \in \mathrm{Spec} (R)$ is an associated prime ideal of $\mathfrak{M}$ if $P= \mathrm{Ann}_{R} (x)$ for some $x \in \mathfrak{M}$, i.e., $P$ is the annihilator ideal of $x$. Let $\mathrm{Ass}_{R} (\mathfrak{M})$ be the set of associated prime ideals of $\mathfrak{M}$, then $\mathrm{Ass}_{R} (\mathfrak{M})$ is a nonempty finite subset of $\mathrm{Spec} (R)$ (cf. Section \ref{sec_algebra}).

\begin{bigthm}\label{thm_noether}
Let $R$, $S$ and $\mathfrak{M}$ be defined as the above. The following statements are equivalent.
\begin{enumerate}
\item The pair $(\mathfrak{M}, P)$ satisfies Property \ref{pt_eigenvalue} for any $P \in \mathrm{Spec}(R)$.
\item The pair $(\mathfrak{M}, P)$ satisfies Property \ref{pt_eigenvalue} for any $P \in \mathrm{Ass}_{R} (\mathfrak{M})$.
\item As an $R$-module, $\mathfrak{M}$ is finitely generated.
\end{enumerate}
\end{bigthm}

\medskip

\noindent \textbf{Outline.} In Section \ref{sec_algebra}, we prove Theorem \ref{thm_noether} which is the algebraic foundation of this paper. In Section \ref{sec_cw_homology}, the homological finiteness result Theorem \ref{thm_cw} is proved. With this preparation Theorem \ref{thm_cw_homotopy} is proved in Section \ref{sec_cw_homotopy}. Section \ref{sec_fiber} includes the proof of the fibering results of self-covering manifolds such as Theorems \ref{thm_high_manifold} and \ref{thm_4_top} and Corollaries  \ref{cor_period}, \ref{cor_5_top} and \ref{cor_monodromy}. On the other hand, the non-fibering Theorems \ref{thm_no_finite}, \ref{thm_no_fiber} and \ref{thm_5_diff} are proved in Section \ref{sec_non-fiber}. Finally, we propose some questions in Section \ref{sec_question}.

\medskip

\noindent \textbf{Acknowledgement.} We are indebted to Francis Thomas Farrell for his generous support over the years. He developed an obstruction theory of fibering manifolds over a circle in his PhD thesis \cite{Farrell} over fifty years ago, which has deeply influenced the present work. This paper is dedicated to him on the occasion of his 80th birthday.

We thank Chuangxun Cheng, John R. Klein, Linquan Ma, Xiaolei Wu, Li Yu and Ruixiang Zhang for various discussions. We also thank Simon K. Donaldson and Alex Waldron who informed us that Question \ref{que_4_diff} is still open. We express our gratitude to an anonymous referee who meticulously read the paper and made many helpful suggestions which led to an improved presentation. The first author is partially supported by NSFC11871272, the second author is partially supported by NSFC12071462, the third author is partially supported by a Sloan fellowship.

\section{Algebraic Foundation}\label{sec_algebra}
In this section, we prove Theorem \ref{thm_noether} which is the algebraic foundation of this paper. The proof uses extensively results in commutative algebra.

By our assumption, both $R$ and $S$ are Noetherian commutative rings, and $S$ is an $R$-algebra via the ring homomorphism $\iota: R \rightarrow S$ and $S$ is generated by $t_{1}, \dots, t_{n}$ as an $R$-algebra. Since $\mathfrak{M}$ is a nonzero finitely generated $S$-module, the set $\mathrm{Ass}_{S} (\mathfrak{M})$ is nonempty and finite (cf. \cite[Theorems~6.1~\&~6.5]{Matsumura2}). By \cite[(9.A)]{Matsumura1} (or \cite[Exercise~6.7]{Matsumura2}),
\[
\mathrm{Ass}_{R} (\mathfrak{M}) = \{ \iota^{-1} (Q) \mid Q \in \mathrm{Ass}_{S} (\mathfrak{M}) \}.
\]
Thus $\mathrm{Ass}_{R} (\mathfrak{M})$ is also nonempty and finite.

The difficult part of Theorem \ref{thm_noether} is the implication $(2) \Rightarrow (3)$. To prove it, we shall reduce $\mathfrak{M}$ to be a module with one generator over $S$. So we first discuss some quotient rings of $S$.

\begin{lemma}\label{lem_prime}
Suppose $Q \in \mathrm{Spec} (S)$ and $P= \iota^{-1} (Q)$. If $(S/Q, P)$ satisfies Property~\ref{pt_eigenvalue}, then $S/Q$ is a finitely generated $R$-module.
\end{lemma}
\begin{proof}
Since $\kappa (P) \otimes_{R} (S/Q)$ is finite-dimensional over $\kappa (P)$ and the eigenvalues of $t_{i}$ on it are integral over $R/P$, there is a monic polynomial $f_{i} (u) \in (R/P) [u]$ such that $f_{i} (t_{i}) =0$ on this vector space. Thus we have the following commutative diagram.
\[
\xymatrix{
S/Q \ar[d] \ar[r]^{\times f_{i}(t_{i})} & S/Q \ar[d] \\
\kappa (P) \otimes_{R} (S/Q) \ar[r]^{0} & \kappa (P) \otimes_{R} (S/Q)
}
\]
Here the vertical maps are the map $x \mapsto 1 \otimes x$. Since $Q$ is prime, $S/Q$ is a domain and hence a torsion free module over $R/P$. Notice that
\[
\kappa (P) \otimes_{R} (S/Q) = \kappa (P) \otimes_{R/P} (S/Q).
\]
Thus these vertical maps are injective, which implies that the action of $f_{i} (t_{i})$ on $S/Q$ is $0$. Let $\bar{t}_{i}$ be the image of $t_{i}$ in $S/Q$. We see $f_{i} (\bar{t}_{i}) =0$ in $S/Q$. Thus $\bar{t}_{i}$ is integral over $R/P$. Since $S/Q$ is an $R/P$-algebra with generators $\bar{t}_{i}$, $1 \leq i \leq n$, we infer $S/Q$ is a finitely generated $R/P$-module (\cite[Corollary~5.2]{Atiyah_Macdonald}), which implies the conclusion.
\end{proof}

\begin{remark}
Actually, it's unnecessary to assume the Noetherianity of $R$ and $S$ in Lemma \ref{lem_prime}.
\end{remark}

Let $I$ be an ideal of $S$. Recall that the radical of $I$ is another ideal of $S$, which is defined as
\[
\sqrt{I} = \{ s \in S \mid \text{$s^{n} \in I$ for some $n>0$}\}.
\]
\begin{lemma}\label{lem_radical}
Let $I$ be an ideal of $S$. If $S/ \sqrt{I}$ is a finitely generated $R$-module, then so is $S/I$.
\end{lemma}
\begin{proof}
Let $J$ denote $\sqrt{I}$. It suffices to prove that $J/I$ is a finitely generated $R$-module. Since $S$ is Noetherian, $J^{k}$ is a finitely generated $S$-module for all $k > 0$. Hence $J^{k} / J^{k+1}$ is a finitely generated $S/J$-module. As $S/J$ is a finitely generated $R$-module, we infer that $J^{k} / J^{k+1}$ is as well. By the Noetherianity of $S$ again, $J^{k} \subseteq I$ when $k$ is large enough. Therefore, $J/I$ is the quotient of an iterated extension of finitely many finitely generated $R$-modules $J^{k} / J^{k+1}$, and hence $J/I$ is a finitely generated $R$-module.
\end{proof}

\begin{lemma}\label{lem_sub_eigenvalue}
Suppose $\mathfrak{L}$ is an $S$-module and $Q \in \mathrm{Spec} (S)$ such that $Q \subseteq \sqrt{\mathrm{Ann}_{S} (\mathfrak{L})}$. Let $P= \iota^{-1} (Q)$. Suppose further $(\mathfrak{L}, P)$ satisfies Property \ref{pt_eigenvalue}. Then $(\mathfrak{N}, P)$ satisfies Property \ref{pt_eigenvalue} for every $S$-submodule $\mathfrak{N}$ of $\mathfrak{L}$.
\end{lemma}
\begin{proof}
Firstly, we show that $(Q^{k} \mathfrak{L}/ Q^{k+1} \mathfrak{L}, P)$ satisfies Property \ref{pt_eigenvalue} for each $k \geq 0$.

By the right exactness of the tensor product, the epimorphism $\mathfrak{L} \rightarrow \mathfrak{L}/Q \mathfrak{L}$ yields an epimorphism
\[
\kappa (P) \otimes_{R} \mathfrak{L} \rightarrow \kappa (P) \otimes_{R} (\mathfrak{L} /Q \mathfrak{L}).
\]
Since $(\mathfrak{L}, P)$ satisfies Property \ref{pt_eigenvalue}, we proved the claim in the case of $k=0$. Here we use the fact that the set of eigenvalues of $t_{i}$ on $\kappa (P) \otimes_{R} (\mathfrak{L}/Q \mathfrak{L})$ is a subset of that of $t_{i}$ on $\kappa (P) \otimes_{R} \mathfrak{L}$.

Since $S$ is Noetherian, $Q$ is a finitely generated $S$-module with some generators $q_{1}, \dots, q_{m}$. We obtain an epimorphism
\begin{eqnarray*}
\bigoplus_{i=1}^{m} Q^{k} \mathfrak{L}/ Q^{k+1} \mathfrak{L} & \rightarrow & Q^{k+1} \mathfrak{L}/ Q^{k+2} \mathfrak{L}, \\
(x_{1}, \dots, x_{m}) & \mapsto & \sum_{i=1}^{m} q_{i} x_{i}.
\end{eqnarray*}
Our first claim is proved by an induction on $k$.

Secondly, define $\mathfrak{N}_{k} = \mathfrak{N} \cap Q^{k} \mathfrak{L}$ for all $k \geq 0$. Here $\mathfrak{N}_{0} = \mathfrak{N}$. Since $S$ is Noetherian and $Q \subseteq \sqrt{\mathrm{Ann}_{S} (\mathfrak{L})}$, when $k$ is large enough, we have $Q^{k} \subseteq \mathrm{Ann}_{S} (\mathfrak{L})$, and hence $Q^{k} \mathfrak{L} =0$ and $\mathfrak{N}_{k} =0$ . Clearly, we have the following exact sequence of $S$-modules
\[
0 \rightarrow \mathfrak{N}_{k} / \mathfrak{N}_{k+1} \rightarrow Q^{k} \mathfrak{L}/ Q^{k+1} \mathfrak{L}.
\]
Note that this is actually an exact sequence of $S/Q$-modules and hence  $R/P$-modules. Since $\kappa (P)$ is flat over $R/P$, we have an exact sequence
\[
0 \rightarrow \kappa (P) \otimes_{R/P} (\mathfrak{N}_{k} / \mathfrak{N}_{k+1}) \rightarrow \kappa (P) \otimes_{R/P} (Q^{k} \mathfrak{L}/ Q^{k+1} \mathfrak{L}),
\]
which is actually the the same as
\[
0 \rightarrow \kappa (P) \otimes_{R} (\mathfrak{N}_{k} / \mathfrak{N}_{k+1}) \rightarrow \kappa (P) \otimes_{R} (Q^{k} \mathfrak{L}/ Q^{k+1} \mathfrak{L}).
\]
We infer $(\mathfrak{N}_{k} / \mathfrak{N}_{k+1}, P)$ satisfies Property \ref{pt_eigenvalue}, because so does $(Q^{k} \mathfrak{L}/ Q^{k+1} \mathfrak{L}, P)$.

By the right exactness of the tensor product again, the exact sequence
\[
\mathfrak{N}_{k+1} \rightarrow \mathfrak{N}_{k} \rightarrow \mathfrak{N}_{k} / \mathfrak{N}_{k+1} \rightarrow 0
\]
induces an exact sequence
\[
\kappa (P) \otimes_{R} \mathfrak{N}_{k+1} \rightarrow \kappa (P) \otimes_{R} \mathfrak{N}_{k} \rightarrow \kappa (P) \otimes_{R} (\mathfrak{N}_{k} / \mathfrak{N}_{k+1}) \rightarrow 0.
\]
Since $\mathfrak{N}_{k} =0$ when $k$ is large enough, by a decreasing induction on $k$, we see $(\mathfrak{N},P)$ satisfies Property \ref{pt_eigenvalue}.
\end{proof}

\begin{proposition}\label{prop_associated}
Suppose $\mathfrak{L}$ is a finitely generated $S$-module such that $\mathrm{Ass}_{S} (\mathfrak{L}) = \{ Q \}$. Let $P= \iota^{-1} (Q)$. Suppose further $(\mathfrak{L},P)$ satisfies Property \ref{pt_eigenvalue}. Then $\mathfrak{L}$ is a finitely generated $R$-module.
\end{proposition}
\begin{proof}
Firstly, we claim that $S/Q$ is a finitely generated $R$-module. By definition, there exists $x \in \mathfrak{L}$ such that $\mathrm{Ann}_{S}(x) = Q$. Thus, we have a monomorphism of $S$-modules:
\begin{eqnarray*}
S/Q & \rightarrow & \mathfrak{L}, \\
s+Q & \mapsto & sx.
\end{eqnarray*}
Furthermore, since $\mathfrak{L}$ is a finitely generated $S$-module and $\mathrm{Ass}_{S} (\mathfrak{L}) = \{ Q \}$, we see $\sqrt{\mathrm{Ann}_{S} (\mathfrak{L})} = Q$ (\cite[Theorem~6.6]{Matsumura2}). As $(\mathfrak{L}, P)$ satisfies Property \ref{pt_eigenvalue}, by Lemma \ref{lem_sub_eigenvalue}, we further infer $(S/Q, P)$ satisfies Property \ref{pt_eigenvalue}. Now our claim follows from Lemma~\ref{lem_prime}.

By assumption, $\mathfrak{L}$ is a finitely generated nonzero $S$-module. In other words, $\mathfrak{L} = \sum_{i=1}^{m} S x_{i}$ for some nonzero $x_{i} \in \mathfrak{L}$. To finish the proof, it suffices to show $S/ I_{i} \cong S x_{i}$ is a finitely generated $R$-module for each $i$, where $I_{i} = \mathrm{Ann}_{S} (x_{i})$. Since $S x_{i} \subseteq \mathfrak{L}$, by the definition of associated prime ideals, we have
\[
\mathrm{Ass}_{S} (S/ I_{i}) = \mathrm{Ass}_{S} (S x_{i}) \subseteq \mathrm{Ass}_{S} (\mathfrak{L}) = \{ Q \}.
\]
As $S/ I_{i} \neq 0$, the set $\mathrm{Ass}_{S} (S/ I_{i})$ is nonempty (\cite[Theorem~6.1]{Matsumura2}). Thus $\mathrm{Ass}_{S} (S/ I_{i}) = \{ Q \}$, which further implies $\sqrt{I_{i}} = Q$ (\cite[Theorem~6.6]{Matsumura2}). As we already know $S/Q$ is a finitely generated $R$-module, by Lemma \ref{lem_radical}, $S/ I_{i}$ is also a finitely generated $R$-module, which finishes the proof.
\end{proof}

We are ready to prove Theorem \ref{thm_noether}.
\begin{proof}[Proof of Theorem \ref{thm_noether}]
$(1) \Rightarrow (2)$. This is trivial since $\mathrm{Ass}_{R} (\mathfrak{M}) \subseteq \mathrm{Spec} (R)$.

\medskip

$(2) \Rightarrow (3)$. Since $S$ is Noetherian and $\mathfrak{M}$ is a finitely generated nonzero $S$-module, by the primary decomposition of the zero submodule (\cite[Theorem~6.8]{Matsumura2}), there exist $S$-submodules $\mathfrak{M}_{i}$ of $\mathfrak{M}$, $1 \leq i \leq k$, such that
\begin{equation}\label{thm_noether_1}
0 = \bigcap_{i=1}^{k} \mathfrak{M}_{i},
\end{equation}
and $\mathrm{Ass}_{S} (\mathfrak{M}/ \mathfrak{M}_{i}) = \{ Q_{i} \}$ for some $Q_{i} \in \mathrm{Ass}_{S} (\mathfrak{M})$.

By \cite[(9.A)]{Matsumura1} (or \cite[Exercise~6.7]{Matsumura2}), for each $Q_{i}$, we have $P_{i} = \iota^{-1} (Q_{i}) \in \mathrm{Ass}_{R} (\mathfrak{M})$. Thus, by assumption, $(\mathfrak{M}, P_{i})$ satisfies Property \ref{pt_eigenvalue}. Since the tensor product is right exact, the projection $\mathfrak{M} \rightarrow \mathfrak{M}/ \mathfrak{M}_{i}$ induces an epimorphism
\[
\kappa (P_{i}) \otimes_{R} \mathfrak{M} \rightarrow \kappa (P_{i}) \otimes_{R} (\mathfrak{M}/ \mathfrak{M}_{i}).
\]
Thus $(\mathfrak{M}/ \mathfrak{M}_{i}, P_{i})$ satisfies Property \ref{pt_eigenvalue}. By Proposition \ref{prop_associated}, we have $\mathfrak{M}/ \mathfrak{M}_{i}$ is a finitely generated $R$-module.

By (\ref{thm_noether_1}), the diagonal homomorphism
\[
\mathfrak{M} \rightarrow \bigoplus_{j=1}^{k} \mathfrak{M}/ \mathfrak{M}_{i}
\]
is injective. Since each $\mathfrak{M}/ \mathfrak{M}_{i}$ is a finitely generated $R$-module and $R$ is Noetherian, $\mathfrak{M}$ is also a finitely generated $R$-module.

\medskip

$(3) \Rightarrow (1)$. Let $P \in \mathrm{Spec} (R)$. Since $\mathfrak{M}$ is a finitely generated $R$-module, as a $\kappa (P)$-vector space, $\kappa (P) \otimes_{R} \mathfrak{M}$ is finite dimensional. Furthermore, by (\cite[Proposition~2.4]{Atiyah_Macdonald}), there is a monic polynomial $f_{i} (u) \in R [u]$ such that the action of $f_{i} (t_{i})$ on $\mathfrak{M}$ is zero. Thus, the action of $f_{i} (t_{i})$ on $\kappa (P) \otimes_{R} \mathfrak{M}$ is also zero. In particular, the eigenvalues of the induced $t_{i}$-action on $\kappa (P) \otimes_{R} \mathfrak{M}$ are integral over $R/P$.
\end{proof}

\begin{remark}
Actually, the implication $(3) \Rightarrow (1)$ in Theorem \ref{thm_noether} does not need the assumption that $R$ and $S$ are Noetherian.
\end{remark}

To prove Theorem \ref{thm_cw}, we also need the following Proposition \ref{prop_eigenvalue} on eigenvalues.

\begin{proposition}\label{prop_eigenvalue}
Let $R$ be a Noetherian domain. Let $K$ be the quotient field of $R$. Let $\sigma$ be an automorphism of $R$. By abusing notations, we also use $\sigma$ to denote the induced automorphism of $K$. Suppose $A$ and $B$ are square matrices with entries in $K$, and $B$ is invertible. Suppose further there exist an integer $n>1$ and a unit $u \in R$ such that
\begin{equation}\label{prop_eigenvalue_1}
\sigma \left( A \right) = u B A^{n} B^{-1}.
\end{equation}
Then, for each nonzero eigenvalue $\lambda$ of $A$, both $\lambda$ and $\lambda^{-1}$ are integral over $R$.
\end{proposition}

Let $L$ be the splitting field of the characteristic polynomial of $A$. Then $L/K$ is a finite algebraic extension.

\begin{lemma}
$\sigma$ extends to be an automorphism of $L$.
\end{lemma}
\begin{proof}
Let $K'$ be the algebraic closure of $K$. Then $\sigma$ extends to be an automorphism of $K'$. Considering $L$ as a subfield of $K'$, it suffices to show that $\sigma (L) = L$. Let $\lambda$ be an eigenvalue of $A$. Clearly, $\sigma (\lambda)$ is an eigenvalue of $\sigma (A)$. By (\ref{prop_eigenvalue_1}), there is an eigenvalue $\lambda_{1}$ of $A$ such that $\sigma (\lambda) = u \lambda_{1}^{n}$. Since $u \in L$ and $\lambda_{1} \in L$, we see $\sigma (\lambda) \in L$. Thus $\sigma (L) \subseteq L$. Since $[L:K] = [\sigma (L): K]$ is finite, we obtain $\sigma (L) = L$.
\end{proof}

Let $\overline{R}$ be the integral closure of $R$ in $L$. Then $\overline{R}$ is a domain (\cite[Corollary~5.3]{Atiyah_Macdonald}). By the argument as in the proof of \cite[Proposition~5.12]{Atiyah_Macdonald}, we immediately get the following lemma.
\begin{lemma}
The field $L$ is the quotient field of $\overline{R}$.
\end{lemma}

Furthermore, by the Mori-Nagata Theorem (\cite[Proposition~6]{Nishimura}, see also \cite[(33.10)]{Nagata}), $\overline{R}$ is a Krull domain. By definition, $\overline{R}$ being a Krull domain means there exists a family $\{ v_{i} \mid i \in \mathfrak{I} \}$ of discrete valuations $v_{i}: L \setminus \{ 0 \} \rightarrow \mathbb{Z}$ such that
\begin{equation}\label{eqn_DVR}
\overline{R} \setminus \{ 0 \} = \bigcap_{i \in \mathfrak{I}} v_{i}^{-1} ([0, +\infty))
\end{equation}
and, for each $x \in L \setminus \{ 0 \}$, $v_{i} (x) = 0$ for all but finitely many $i \in \mathfrak{I}$. See \cite[\S~12]{Matsumura2} for more details about Krull domains, and see \cite[p.~94]{Atiyah_Macdonald} for the definition of a discrete valuation.

\begin{lemma}\label{lem_matrix_eigenvalue}
For each eigenvalue $\lambda$ of $A$ and each integer $m>0$, there exist an $x_{m} \in L$ and a unit $u_{m} \in R$ such that $\lambda = u_{m} x_{m}^{n^{m}}$.
\end{lemma}
\begin{proof}
Applying (\ref{prop_eigenvalue_1}) repeatedly, we obtain
\[
\sigma^{2} (A) = \sigma (u) \sigma (B) \sigma (A)^{n} \sigma (B)^{-1} = \sigma (u) \sigma (B) u^{n} B A^{n^{2}} B^{-1} \sigma (B)^{-1}
\]
and hence
\[
A = u_{2} B_{2} \cdot \sigma^{-2} \left( A^{n^{2}} \right) \cdot B_{2}^{-1},
\]
where $u_{2} = \sigma^{-1} (u) \sigma^{-2} \left( u^{n} \right)$ is a unit of $R$ and $B_{2} = \sigma^{-1} (B) \sigma^{-2} (B)$. By an induction on $m$, we obtain
\[
A = u_{m} B_{m} \cdot \sigma^{-m} \left( A^{n^{m}} \right) \cdot B_{m}^{-1}
\]
with $u_{m}$ a unit of $R$. Now the desired conclusion follows.
\end{proof}

\begin{proof}[Proof of Proposition \ref{prop_eigenvalue}]
Suppose $\lambda \neq 0$ is an eigenvalue of $A$. Let $v_{i}$ be a discrete valuation in (\ref{eqn_DVR}). By Lemma \ref{lem_matrix_eigenvalue}, for each integer $m>0$, we have $\lambda = u_{m} x_{m}^{n^{m}}$ for a unit $u_{m} \in R$ and $x_{m} \in L$. Then
\[
v_{i} (\lambda) = v_{i} (u_{m}) + n^{m} v_{i} (x_{m}) = 0 + n^{m} v_{i} (x_{m}) = n^{m} v_{i} (x_{m}).
\]
Thus $n^{m}$ divides $v_{i} (\lambda)$ for each $m>0$. Since $n>1$, we infer $v_{i} (\lambda) =0$ which also implies $v_{i} (\lambda^{-1}) =0$. By (\ref{eqn_DVR}), both $\lambda$ and $\lambda^{-1}$ are in $\overline{R}$, which finishes the proof.
\end{proof}

\section{CW Complex: Homological Finiteness}\label{sec_cw_homology}
In this section, we prove a homological finiteness theorem which is the algebraic topological foundation of this work. Let us first explain its assumption, which is more general than that of Theorem \ref{thm_cw_homotopy}.

Let $X$ be a connected CW complex. Suppose there is an epimorphism
\begin{equation}\label{eqn_pi1}
\varphi: \  \pi_{1} (X) \rightarrow G \times \mathbb{Z},
\end{equation}
where $G$ is a finitely generated abelian group. Let $X_{k}$ be the $k$-cover of $X$ with $\pi_{1} (X_{k}) = \varphi^{-1} (G \times k\mathbb{Z})$, and $X_{\infty}$ be the infinite cyclic cover of $X$ with $\pi_{1} (X_{\infty}) = \varphi^{-1} (G \times 0)$.  Assume that there is a homotopy equivalence $h: X \rightarrow X_{k}$ for some $k>1$, such that the induced isomorphism $
h_{\sharp}: \pi_{1} (X) \overset{\cong}{\longrightarrow} \pi_{1} (X_{k})
$
satisfies
\begin{equation}\label{eqn_assumption1}
h_{\sharp} (\ker \varphi) = \ker \varphi \quad \text{and} \quad h_{\sharp} (\varphi^{-1} (G \times 0)) = \varphi^{-1} (G \times 0).
\end{equation}
In the special case that $\varphi = \mathrm{id}$, the space $X$ satisfies the assumption of Theorem \ref{thm_cw_homotopy}.

By (\ref{eqn_assumption1}), the above $h_{\sharp}$ induces an isomorphism
\begin{equation}\label{eqn_assumption2}
[h_{\sharp}] \colon G \times \mathbb{Z} \rightarrow G \times (k\mathbb{Z}) \quad \text{with} \quad [h_{\sharp}] (G\times 0) = G \times 0.
\end{equation}
The epimorphism $\varphi \colon \pi_1(X_{\infty}) \to G$ gives rise to an action
\begin{equation}\label{eqn_local_system}
\pi_{1} (X_{\infty}) \times G  \rightarrow  G,  \  \ (\xi, g)  \mapsto  \varphi (\xi) \cdot g.
\end{equation}
Let $R_0$ be a commutative Noetherian ring,  and let $R=R_0[G]$ be the group ring. We identify the group ring $R_0[G \times \mathbb Z]$ with the Laurent polynomial ring $R[t,t^{-1}]$, where $t$ corresponds to $1 \in \mathbb Z$. Both $R$ and $R[t, t^{-1}]$ are Noetherian commutative rings.  Under the action of (\ref{eqn_local_system}), $R$ is a $\pi_1(X_{\infty})$-module, and $t$ acts on $X_{\infty}$ via deck transformation. The cellular chain complex $C_{\bullet}(X_{\infty};R)$ with local coefficient system $R$ is a chain complex of free $R[t, t^{-1}]$-modules, with a basis corresponding to the cells of $X$.

\begin{proposition}\label{prop_finite}
If $X$ is of finite type, then the homology group $H_{j} (X_{\infty}; R)$ is a finitely generated $R[t, t^{-1}]$-module for each $j$.
\end{proposition}
\begin{proof}
Since $X$ is of finite type, $C_{j} (X_{\infty}; R)$ is a free $R[t, t^{-1}]$-module of finite rank for each $j$. The conclusion follows from the Noetherianity of $R[t, t^{-1}]$.
\end{proof}

The main result of this section is the following stronger finiteness theorem.
\begin{theorem}\label{thm_cw}
 If $X$ is of finite type, then the homology group $H_{j} (X_{\infty}; R)$ is a finitely generated $R$-module for each $j$.
 \end{theorem}

We prove this theorem in the rest of this section.

For each $P \in \mathrm{Spec}(R)$, let $\kappa (P)$ be the residual field of $R$ at $P$, i.e. the quotient field of $R/P$. We first study the eigenvalues of the action of the deck transformation $t$ on $\kappa (P) \otimes_{R} H_{j} (X_{\infty}; R)$.

Lifting the homotopy equivalence $h \colon X \rightarrow X_{k}$ to the covering space $X_{\infty}$, one gets a homotopy equivalence $\overline{h} \colon X_{\infty} \rightarrow X_{\infty}$. By assumption (\ref{eqn_assumption2}), $\overline{h}_{\sharp} \colon \pi_{1} (X_{\infty}) \rightarrow \pi_{1} (X_{\infty})$ induces an automorphism $
[h_{\sharp}]|_{G \times 0} \colon \ G \rightarrow G$. Extend the isomorphism $[h_{\sharp}] \colon G \times 0 \to G \times 0$ in (\ref{eqn_assumption2}) to an automorphism $\sigma$ of the group ring $R=R_0[G]$:
\[
\sigma \left( \sum_{i} a_{i} g_{i} \right) = \sum_{i} a_{i} [h_{\sharp}] (g_{i}),
\]
where $a_{i} \in R_{0}$ and $g_{i} \in G$. Then the homotopy equivalence $\overline{h}$ induces a $\sigma$-twisted $R$-module isomorphism $\overline{h}_{*} \colon H_{j} (X_{\infty}; R) \rightarrow H_{j} (X_{\infty}; R)$. Here $\sigma$-twisted means, for any $r \in R$ and $x \in H_{j} (X_{\infty}; R)$, the following equality always holds
\begin{equation}\label{eqn_sigma_linear}
\overline{h}_{*} (rx) = \sigma (r) \overline{h}_{*} (x).
\end{equation}

Next we examine the commutativity between $\overline h_*$ and the action of $t$ on $H_{j} (X_{\infty}; R)$. The isomorphism $[h_{\sharp}]$ in (\ref{eqn_assumption2}) has the form
\begin{equation}\label{eqn_deck1}
[h_{\sharp}] \colon \ G \times \mathbb{Z}  \rightarrow  G \times \mathbb{Z}, \ \ \
(g,n)  \mapsto  (\sigma (g) + n g_{0}, \pm kn),
\end{equation}
where $[h_{\sharp}] (0,1) = (g_{0}, \pm k)$. Generalizing Definition \ref{def_deck}, we say $h$ \emph{preserves the orientation of deck transformations} if $[h_{\sharp}] (0,1) = (g_{0}, k)$ and $h$ \emph{reverses the orientation of deck transformations} if otherwise.
From  (\ref{eqn_deck1}), we see that the chain map $\overline h_* \colon C_{\bullet} (X_{\infty}; R) \to C_{\bullet} (X_{\infty}; R)$ satisfies \begin{equation}\label{eqn_conjugate}
\overline{h}_{*} \cdot t = g_{0} \cdot t^{\pm k} \cdot \overline{h}_{*},
\end{equation}
and hence the same equality holds for $\overline h_* \colon H_{\bullet} (X_{\infty}; R) \to H_{\bullet} (X_{\infty}; R) $.

By Proposition \ref{prop_finite} and \cite[Theorem 6.1, 6.5]{Matsumura2}, $\mathrm{Ass}_{R[t, t^{-1}]} (H_{j} (X_{\infty}; R))$ is a nonempty finite set provided that $H_{j} (X_{\infty}; R)$ is nonzero. Also by \cite[(9.A)]{Matsumura1},
\[
\mathrm{Ass}_{R} (H_{j} (X_{\infty};R)) = \{ Q \cap R \mid Q \in \mathrm{Ass}_{R[t, t^{-1}]} (H_{j} (X_{\infty}; R)) \},
\]
which is also finite.

\begin{lemma}\label{lem_eigenvalue}
Suppose $P \in \mathrm{Ass}_{R} (H_{j} (X_{\infty}; R))$, and $\kappa (P) \otimes_{R} H_{j} (X_{\infty}; R)$ is a nonzero finite-dimensional $\kappa (P)$-vector space. Then the eigenvalues of $t$ and $t^{-1}$ actions on this vector space are integral over $R/P$.
\end{lemma}
\begin{proof}
Lifting to covering spaces, the homotopy equivalence $h \colon X \to X_{k}$ induces a family of homotopy equivalences $X_{k^{i-1}} \to X_{k^{i}}$ for all $i >0$. Composing these maps we obtain homotopy equivalences $h_{i} \colon X \to X_{k^{i}}$ for all $i >0$. Note that $h_{i} \colon X \to X_{k^{i}}$ lifts to a homotopy equivalence $X_{\infty} \to X_{\infty}$ which is exactly $\overline h^{i}$, the $i$-fold composition of $\overline h$.

By (\ref{eqn_conjugate}), we have $\overline{h}_{*}^{2} \cdot t = g_{1} \cdot t^{k^{2}} \cdot \overline{h}_{*}^{2}$ for some $g_{1} \in G$. Replacing $\overline{h}_{*}$ with $\overline{h}_{*}^{2}$ if necessary, we may assume $h$ preserves the orientation of deck transformation.

We know $\overline{h}_{*} \colon H_{j} (X_{\infty}; R) \rightarrow H_{j} (X_{\infty}; R)$ is a $\sigma$-twisted isomorphism of $R$-modules (c.f. (\ref{eqn_sigma_linear})) and $\sigma$ is an automorphism of $R$. It follows from the definition of the associated prime ideals that if $P' \in \mathrm{Ass}_{R} (H_{j} (X_{\infty}; R))$, then so are $\sigma (P')$ and $\sigma^{-1} (P')$. Thus, $\sigma$ acts on the finite set $\mathrm{Ass}_{R} (H_{j} (X_{\infty}; R))$. Then $\sigma^{i}$ acts trivially on $\mathrm{Ass}_{R} (H_{j} (X_{\infty}; R))$ for some $i>0$. Note that $\overline{h}_{*}^{i}$ is $\sigma^{i}$-twisted. Replacing $\overline{h}_{*}$ with $\overline{h}_{*}^{i}$ if necessary, we may assume $\sigma (P) = P$. In this case, $\sigma$ induces automorphisms of $R/P$ and $\kappa (P)$.

The $\sigma$-twisted $R$-module isomorphism $\overline h_* \colon  H_{j} (X_{\infty}; R) \rightarrow H_{j} (X_{\infty}; R)$ extends to a well-defined map $\sigma \otimes \overline h_* \colon  \kappa (P) \otimes_{R} H_{j} (X_{\infty}; R) \rightarrow \kappa (P) \otimes_{R} H_{j} (X_{\infty}; R)$. By (\ref{eqn_conjugate}), we have the following equality
\[
t = (\sigma \otimes \overline{h}_{*})^{-1} \cdot g_{0} \cdot t^{k} \cdot (\sigma \otimes \overline{h}_{*}),
\]
of actions on $\kappa (P) \otimes_{R} H_{j} (X_{\infty}; R)$, where $t$ and $g_0$ stand for $\kappa (P)$-linear maps $1 \otimes t$ and $ 1 \otimes g_0$ respectively. Since $(\sigma \otimes \overline{h}_{*})^{-1}$ is $\sigma^{-1}$-twisted, we have
\begin{equation}\label{lem_eigenvalue_1}
t = \sigma^{-1} (g_{0}) \cdot (\sigma \otimes \overline{h}_{*})^{-1} \cdot t^{k} \cdot (\sigma \otimes \overline{h}_{*}).
\end{equation}

The map $\sigma \otimes \overline h_*$ is not $\kappa (P)$-linear, to obtain a $\kappa (P)$-linear transformation, we proceed as follows.
By assumption, $ \kappa (P) \otimes_{R} H_{j} (X_{\infty}; R)$ is a finite dimensional $\kappa (P)$-vector space. Identify this vector space with $\kappa (P)^{m}$ for some $m > 0$ and view $t$, $g_0$, $\sigma \otimes \overline h_*$ as maps on $\kappa (P)^m$. Define a map
\[
\Phi_{\sigma}  \colon \kappa (P)^{m} \to \kappa(P)^m, \ \ \ (x_{1}, \dots, x_{m}) \mapsto (\sigma (x_{1}), \dots, \sigma (x_{m})).
\]
Then, by (\ref{lem_eigenvalue_1}), we have
\begin{eqnarray*}
\Phi_{\sigma} t \Phi_{\sigma}^{-1} & = & g_{0} \cdot \left( \Phi_{\sigma} (\sigma \otimes \overline{h}_{*})^{-1} \right) \cdot t^{k} \cdot \left( (\sigma \otimes \overline{h}_{*}) \Phi_{\sigma}^{-1} \right) \\
& = & g_{0} \cdot \left( \Phi_{\sigma} (\sigma \otimes \overline{h}_{*})^{-1} \right) \cdot t^{k} \cdot \left(\Phi_{\sigma} (\sigma \otimes \overline{h}_{*})^{-1}\right )^{-1}.
\end{eqnarray*}
Now $t$, $\Phi_{\sigma} t \Phi_{\sigma}^{-1}$ and $\Phi_{\sigma}  (\sigma \otimes \overline{h}_{*})^{-1}$ are $\kappa (P)$-linear transformations of $\kappa (P)^m$ and hence can be identified with square matrices over $\kappa (P)$. If $t$ is represented by a matrix $A$, then $\Phi_{\sigma} t \Phi_{\sigma}^{-1}$ is represented by $\sigma (A)$. Note also that $g_{0} \in G$ is a unit in the group ring $R=R_{0}[G]$, hence a unit in $R/P$. We finish the proof by taking $R$, $A$, $B$, $u$ and $n$ in Proposition \ref{prop_eigenvalue} as $R/P$, $t$, $\Phi_{\sigma} (\sigma \otimes \overline{h}_{*})^{-1}$, $g_{0}$ and $k$ respectively.
\end{proof}

\begin{remark}\label{rmk_eigenvalue}
The proof of Lemma \ref{lem_eigenvalue} can be significantly simplified if $G=0$. In this case, $\overline{h}_{*}$ is an $R$-module isomorphism and it is easy to show that the eigenvalues of $t^{\pm 1}$ on $\kappa (P) \otimes_{R} H_{j} (X_{\infty}; R)$ are roots of unity and hence integral over $R/P$ for each $P \in \mathrm{Spec} (R)$, without referring to Proposition \ref{prop_eigenvalue}. See the proof of Theorem \ref{thm_cw_period} (2).
\end{remark}

We view $\kappa (P)$ as a $\pi_1(X)$-,  $\pi_1(X_q)$- and $\pi_1(X_{\infty})$-module via the homomorphisms
$$ \pi_1(X) \stackrel{\varphi}{\rightarrow} G \times \mathbb Z \rightarrow G, \ \   \pi_1(X_q) \stackrel{\varphi}{\rightarrow} G \times q \mathbb Z \rightarrow G \ \   \mathrm{and} \  \ \pi_1(X_{\infty}) \stackrel{\varphi}{
\rightarrow} G$$
respectively.  The cellular chain complex $C_{\bullet}(X_q;\kappa (P)) = \kappa (P) \otimes_R C_{\bullet}(X_q;R)$ and the homology groups $H_j(X_q;\kappa (P))$ are vector spaces over $\kappa (P)$.

\begin{lemma}\label{lem_dimension_bound}
The dimension $\dim_{\kappa (P)} H_{j} (X_{k^{i}}; \kappa (P))$ is less than or equal to the number of $j$-cells of $X$ for any $i>0$ and $P \in \mathrm{Spec} (R)$.
\end{lemma}
\begin{proof}
The homotopy equivalence $h_{i}: X \rightarrow X_{k^{i}}$ induces an isomorphism
\begin{equation}\label{lem_dimension_bound_1}
{h_{i}}_{*}:\ H_{j} (X; h_{i}^{*} \kappa (P)) \overset{\cong}{\longrightarrow} H_{j} (X_{k^{i}}; \kappa (P)),
\end{equation}
where $h_{i}$ pulls back $\kappa (P)$ on $X_{k^{i}}$ to be $h_{i}^{*} \kappa (P)$ on $X$. By (\ref{eqn_assumption2}), the isomorphism ${h_{i}}_{\sharp}: \pi_{1} (X) \rightarrow \pi_{1} (X_{k^{i}})$ induces an isomorphism
\[
G \times \mathbb{Z} \overset{\cong}{\longrightarrow} G \times (k^{i} \mathbb{Z})
\]
such that $G \times 0$ maps onto $G \times 0$. Let $L_{i} \subseteq G \times \mathbb{Z}$ be the preimage of $0 \times (k^{i} \mathbb{Z})$. Then $L_{i}$ is a complement of $G \times 0$ in $G \times \mathbb{Z}$. Let $\widehat{X}$ be the cover of $X$ with $\pi_{1} (\widehat{X}) = \varphi^{-1} (L_{i})$. Then the deck transformation group of $\widehat{X} \rightarrow X$ is $G$. Note that
\[
C_{\bullet} (X; h_{i}^{*} \kappa (P)) = \kappa (P') \otimes_{R} C_{\bullet} (X; h_{i}^{*} R) \qquad \text{and} \qquad C_{\bullet} (X; h_{i}^{*} R) \cong C_{\bullet}(\widehat{X}; R_{0}),
\]
where $P' = \sigma^{-i} (P)$. Since $C_{\bullet}(\widehat{X}; R_{0})$ is a chain complex of free $R$-modules with a basis corresponding to the cells of $X$, we infer $\dim_{\kappa (P')} H_{j} (X; h_{i}^{*} \kappa (P))$ is less than or equal to the number of the $j$-cells of $X$. Note that $\sigma^{i}$ induces a field isomorphism $\kappa (P') \rightarrow \kappa (P)$, and (\ref{lem_dimension_bound_1}) is actually a $\sigma^{i}$-twisted isomorphism. Therefore,
\[
\dim_{\kappa (P)} H_{j} (X_{k^{i}}; \kappa (P)) = \dim_{\kappa (P')} H_{j} (X; h_{i}^{*} \kappa (P)),
\]
which implies the conclusion.
\end{proof}

\begin{lemma}\label{lem_vector_space}
The homology group $H_{j} (X_{\infty}; \kappa (P))$ is a finite-dimensional $\kappa (P)$-vector space for each $j$ and $P \in \mathrm{Spec} (R)$.
\end{lemma}
\begin{proof}
The chain complex $C_{\bullet} (X_{\infty}; \kappa (P)) = \kappa (P) \otimes_{R} C_{\bullet} (X_{\infty}; R)$ is a chain complex of free modules over $\kappa (P) \otimes_{R} R[t, t^{-1}] = \kappa (P) [t, t^{-1}]$ with a basis corresponding to the cells of $X$. As the localization of the principal ideal domain (PID) $\kappa (P) [t]$ at $t$, the ring $\kappa (P) [t, t^{-1}]$ is also a PID (see also the third paragraph in \cite[p.~117]{Milnor68}). Since $X$ is of finite type, $H_{j} (X_{\infty}; \kappa (P))$ is a finitely generated $\kappa (P) [t, t^{-1}]$-module. By the structure theorem for finitely generated modules over a PID, it suffices to show $H_{j} (X_{\infty}; \kappa (P))$ is a $\kappa (P) [t, t^{-1}]$-torsion module.

Associated to the covering map $X_{\infty} \rightarrow X_{q}$, there is a short exact sequence of chain complexes
\[
0 \rightarrow C_{\bullet} (X_{\infty}; \kappa (P)) \overset{\times (t^{q}-1)}{\longrightarrow} C_{\bullet} (X_{\infty}; \kappa (P)) \rightarrow C_{\bullet} (X_{q}; \kappa (P)) \rightarrow 0
\]
which induces the long exact sequence of homology groups
\begin{equation}\label{eqn_homology}
H_{j} (X_{\infty}; \kappa (P)) \overset{\times (t^{q}-1)}{\longrightarrow} H_{j} (X_{\infty}; \kappa (P)) \rightarrow H_{j} (X_{q}; \kappa (P)) \overset{\partial}{\rightarrow} H_{j-1} (X_{\infty}; \kappa (P)),
\end{equation}
in which all maps are $R$-module homomorphisms and the maps $\times (t^q-1)$ are $R[t, t^{-1}]$-module homomorphisms.

By (\ref{eqn_homology}), we obtain the exact sequence
\[
0 \rightarrow H_{j} (X_{\infty}; \kappa (P))/ (t^{q} -1) \rightarrow H_{j} (X_{q}; \kappa (P)).
\]
Suppose $H_{j} (X_{\infty}; \kappa (P))$ is not a $\kappa (P) [t, t^{-1}]$-torsion module, or equivalently, the free part of $H_{j} (X_{\infty}; \kappa (P))$ is nonzero. Then
\[
\dim_{\kappa (P)} H_{j} (X_{q}; \kappa (P)) \geq \dim_{\kappa (P)} (H_{j} (X_{\infty}; \kappa (P))/(t^{q}-1)) \geq q.
\]
Thus $\dim_{\kappa (P)} H_{j} (X_{k^{i}}; \kappa (P))$ is unbounded when $i$ tends to infinity, which contradicts Lemma \ref{lem_dimension_bound}.
\end{proof}

We are ready to finish the proof of Theorem \ref{thm_cw}.
\begin{proof}[Proof of Theorem \ref{thm_cw}]
By Proposition \ref{prop_finite}, $H_{j} (X_{\infty}; R)$ is a finitely generated $R[t, t^{-1}]$-module. We have to prove that it is a finitely generated $R$-module. The argument is an induction on $j$. We shall apply Theorem \ref{thm_noether} by taking $S$ and $\mathfrak{M}$ to be $R[t, t^{-1}]$ and $H_{j} (X_{\infty}; R)$ respectively. We may assume $H_{j} (X_{\infty}; R)$ is nonzero.

For each $P \in \mathrm{Spec} (R)$, since
\[
C_{\bullet} (X_{\infty}; \kappa (P)) = \kappa (P) \otimes_{R} C_{\bullet} (X_{\infty}; R),
\]
we may compare $\kappa (P) \otimes_{R} H_{\bullet} (X_{\infty}; R)$ with $H_{\bullet} (X_{\infty}; \kappa (P))$ by the Homological Universal Coefficient Spectral Sequence (\cite[Theorem~5.5.1]{Godement}). More precisely, choose an $R$-projective resolution of $\kappa (P)$:
\[
\kappa (P) \leftarrow \mathfrak{M}_{0} \leftarrow \mathfrak{M}_{1} \leftarrow \cdots.
\]
Let the double complex $D_{\bullet, \bullet}$ be the tensor product of $\mathfrak{M}_\bullet$ and $C_{\bullet} (X_{\infty}; R)$. In particular, $D_{i, j}=\mathfrak{M}_{i} \otimes_{R} C_{j} (X_{\infty}; R)$. This spectral sequence converges to $H_{\bullet} (X_{\infty}; \kappa (P))$ with $E^{2}$-terms
\[
E^{2}_{i,j} = \mathrm{Tor}_{i} (\kappa (P), H_{j} (X_{\infty}; R)).
\]
In particular, $E^{2}_{0,j} = \kappa (P) \otimes_{R} H_{j} (X_{\infty}; R)$.

For $j=0$ and each $P \in \mathrm{Spec} (R)$, we have
\[
\kappa (P) \otimes_{R} H_{0} (X_{\infty}; R) = E^{2}_{0,0} = E^{\infty}_{0,0} = H_{0} (X_{\infty}; \kappa (P)).
\]
By Lemma \ref{lem_vector_space}, $\kappa (P) \otimes_{R} H_{0} (X_{\infty}; R)$ is finite-dimensional over $\kappa (P)$. By Lemma \ref{lem_eigenvalue}, the eigenvalues of the $t$ and $t^{-1}$ actions on $\kappa (P) \otimes_{R} H_{0} (X_{\infty}; R)$ are integral over $R/P$ if we further have $P \in \mathrm{Ass}_{R} (H_{0} (X_{\infty}; R))$. Thus $H_{0} (X_{\infty}; R)$ is a finitely generated $R$-module by the implication $(2) \Rightarrow (3)$ in Theorem \ref{thm_noether}. We obtain the conclusion for $j=0$.

As inductive hypothesis, we assume $H_{m} (X_{\infty}; R)$ are finitely generated $R$-modules for all $m<j$. By the Noetherianity of $R$, the  $\kappa (P)$-vector spaces $E^{2}_{i,m}$ are finite-dimensional for all $m<j$ and $i \geq 0$. Clearly, the differentials in the spectral sequence are $R$-module homomorphisms and hence $\kappa (P)$-homomorphism. Therefore, $E^{r}_{i,m}$ are finite-dimensional $\kappa (P)$-vector spaces for all $r \geq 2$, $m<j$,  and $i \geq 0$. Since
\[
E^{j+2}_{0,j} = E^{\infty}_{0,j} \subseteq H_{j} (X_{\infty}; \kappa (P)),
\]
by Lemma \ref{lem_vector_space} again, $E^{j+2}_{0,j}$ is finite-dimensional. We have the following exact sequence
\[
E^{r}_{r, j-r+1} \overset{\mathrm{d}_{r}}{\rightarrow} E^{r}_{0,j} \rightarrow E^{r+1}_{0,j} \rightarrow 0
\]
for all $r \geq 2$, where $\mathrm{d}_{r}$ is a differential. Thus $E^{2}_{0,j} = \kappa (P) \otimes_{R} H_{j} (X_{\infty}; R)$ is also finite-dimensional by a decreasing induction on $r$. Applying Lemma \ref{lem_eigenvalue} and Theorem \ref{thm_noether} again, $H_{j} (X_{\infty};R)$ is a finitely generated $R$-module. The inductive argument is completed.
\end{proof}

\section{CW Complex: Homotopy Finiteness and Homological Periodicity}\label{sec_cw_homotopy}
We prove Theorem \ref{thm_cw_homotopy} and a homological periodicity theorem (Theorem \ref{thm_cw_period}) in this section.

\begin{proof}[Proof of Theorem \ref{thm_cw_homotopy}]
If $X$ is homotopy equivalent to a CW complex of finite type, without loss of generality, we may  assume $X$ itself is of finite type. Taking $R_{0} = \mathbb{Z}$ and $\varphi = \mathrm{id}$ in Theorem \ref{thm_cw}, we see that
\[
H_{j} (\widetilde{X}; \mathbb{Z}) = H_{j} (X_{\infty}; \mathbb{Z}[G])
\]
is a finitely generated  $\mathbb{Z}[G]$-module for each $j$. Here $\widetilde{X}$ is the universal cover of $X$. Since $G$ is finitely presented and $\mathbb{Z}[G]$ is a Noetherian ring, by Wall's Theorems A and B in \cite{Wall1}, we infer $X_{\infty}$ is homotopy equivalent to a CW complex of finite type.

If $X$ is finitely dominated, by Wall's Theorems A and F in \cite{Wall1}, $X$ is homotopy equivalent to a CW complex of finite type. By the above arguments, so is $X_{\infty}$. Clearly, $X$ satisfies the condition $\mathrm{D}_{n}$ in Wall's Theorem F for some $n$, hence so does $X_{\infty}$. By Wall's Theorems A and F again, $X_{\infty}$ is finitely dominated.
\end{proof}

In Theorems \ref{thm_cw_homotopy} and \ref{thm_cw}, the action of deck transformation on the homology groups of $X_{\infty}$ is actually periodic. This follows from the theorem below.

\begin{theorem}\label{thm_cw_period}
Let $X$ be a CW complex of finite type, $\varphi \colon \pi_{1}(X) \to \mathbb Z$ be an epimorphism, $X_{k}$ be the finite cyclic cover of $X$ with $\pi_{1}(X_{k}) = \varphi^{-1}(k \mathbb Z)$, and $X_{\infty}$ be the infinite cyclic cover with $\pi_{1}(X_{\infty})=\ker \varphi$. Let $d \colon X_{\infty} \to X_{\infty}$ be the deck transformation corresponding to $1 \in \mathbb Z$. Assume there is a homotopy equivalence $h \colon X \to X_{k}$ for some $k >1$ such that the induce isomorphism $h_{\sharp} \colon \pi_{1}(X) \to \pi_{1}(X_{k})$ satisfies $h_{\sharp}(\ker \varphi)= \ker \varphi$.

Then, for each degree $j$, the following statements hold.
\begin{enumerate}
\item The homology group $H_{j} (X_{\infty}; \mathbb{Z})$ is a finitely generated abelian group.

\item There exist positive integers $m$ prime to $k$, and $l$ such that
\[
d_{*}^{m} =\id: H_{j} (X_{\infty}; \mathbb{Z}) /T \rightarrow H_{j} (X_{\infty}; \mathbb{Z}) /T \quad \text{and} \quad d_{*}^{l} = \id: H_{j} (X_{\infty}; \mathbb{Z})\to H_{j} (X_{\infty}; \mathbb{Z}),
\]
where $T$ is the torsion subgroup of $H_{j} (X_{\infty}; \mathbb{Z})$.
\end{enumerate}
\end{theorem}

\begin{proof}
(1). This follows from Theorem \ref{thm_cw} with $R_{0} = \mathbb{Z}$ and $G=0$.

(2). As in the proof of Theorem \ref{thm_cw}, let $\overline{h}: X_{\infty} \rightarrow X_{\infty}$ be the homotopy equivalence induced from $h$. Then $\overline{h}_{*}^{-1} d_{*}^{k} \overline{h}_{*} = d_{*}^{\pm 1}$ on $H_{j} (X_{\infty}; \mathbb{Z})$. (Here $\overline{h}_{*}$ and $d_{*}$ are $\mathbb{Z}$-module isomorphisms, compare Remark \ref{rmk_eigenvalue}.) The conclusion follows from the algebraic results Proposition \ref{prop_matrix} and Lemma \ref{lem_matrix} below.
\end{proof}

\begin{proposition}\label{prop_matrix}
Suppose $A$ and $B$ are in $\mathrm{GL}(n; \mathbb{Z})$ such that $B A^{k} B^{-1} = A^{\pm 1}$ for some $k>1$. Then  $A^{m} =1$ for some positive $m$ prime to $k$.
\end{proposition}
\begin{proof}
The assumption implies $B^{2} A^{k^{2}} B^{-2} = A$. And $m$ is prime to $k$ if it is prime to $k^{2}$. Therefore by replacing $B$ with $B^{2}$ if necessary, we may assume $B A^{k} B^{-1} = A$.

Let $\lambda$ be an eigenvalue of $A$, then the equation $B A^{k} B^{-1} = A$ implies that $\lambda^{k}$ is an eignvalue of $A$ as well. We claim $\lambda$ is a root of unity. Otherwise, $A$ would have infinitely many distinct eigenvalues $\lambda^{k^{u}}$ for all $u \in \mathbb{N}$, which is impossible. Let $\Lambda (A)$ be the set of the eigenvalues of $A$. Let $H$ be the subgroup of the multiplicative group $\mathbb{C} \setminus \{ 0 \}$ generated by $\Lambda (A)$, then $H$ is finite, say of order $m$. Since $\Lambda (A^{k}) = \Lambda (A)$, the map $z \mapsto z^{k}$ is an automorphism of $H$. Therefore, $m$ is prime to $k$.

Let $D=A^m$, then all eigenvalues of $D$ are $1$. Thus $N=D-1$ is nilpotent. We have to prove $N$ is indeed $0$. Suppose $N^{u} =0$ for some $u>0$. When $k^{v} \geq u$, we have
\[
D^{k^{v}} -1 = (1+N)^{k^{v}} -1 = \sum_{i=1}^{u-1} \frac{k^{v} !}{(k^{v} -i)! i!} N^{i}.
\]
Thus for each $r \in \mathbb{N}$, we have $k^{r}$ divides all entries of $D^{k^{v}} -1$ when $v$ is large enough. On the other hand,
\[
B^{v} \left( D^{k^{v}} -1 \right) B^{-v} = D-1 =N.
\]
So $k^{r}$ divides all entries of $N$ for all $r \in \mathbb{N}$, which implies $N=0$ because $k>1$.
\end{proof}

\begin{lemma}\label{lem_period}
Let $0 \rightarrow G_{1} \rightarrow G_{2} \rightarrow G_{3} \rightarrow 0$ be an exact sequence of abelian groups, where all elements in $G_{1}$ or $G_{3}$ are annihilated by some integer $r >0$. Suppose $A$ is an automorphism of $G_{2}$ which induces the identity on $G_{1}$ and $G_{3}$. Then $A^{r}$ is the identity .
\end{lemma}
\begin{proof}
Let $N=A-1$. Then $\mathrm{Im} N \subseteq G_{1}$ and $N|_{G_{1}} =0$. Thus $N^{2} =0$ and $A^{r} = 1+rN$. Furthermore, $N: G_{2} \rightarrow G_{1}$ factors through $G_{3}$. The conclusion follows.
\end{proof}

\begin{lemma}\label{lem_matrix}
Let $H$ be a finitely generated abelian group and $A$ an automorphism of $H$. Suppose for some $m>0$, $A^{m} =1$ on $H/T$, where $T$ is the torsion subgroup of $H$. Then for some $l>0$, $A^{l} =1$ on $H$.
\end{lemma}
\begin{proof}
The automorphism $A$ induces an automorphism of $T$. Since $T$ is finite, we have $A^{s} =1$ on $T$ for some $s >0$. Thus $A^{ms} =1$ on both $T$ and $H/T$. The conclusion follows from Lemma \ref{lem_period}.
\end{proof}

\section{Manifold: Fibering Theorems}\label{sec_fiber}
We start this section with the proof of Theorem \ref{thm_4_top}. Then in the main body of this section, we prove Theorem \ref{thm_high_manifold} and Corollaries \ref{cor_period}, \ref{cor_5_top} and \ref{cor_monodromy}.

\begin{proof}[Proof of Theorem \ref{thm_4_top}]
Let $M_{\infty}$ be the universal cover of $M$,  then the homology group $H_{2} (M_{\infty}; \mathbb{Z})$ is a free $\mathbb{Z}[\mathbb{Z}]$-module (see e.g. the first paragraph in \cite[p.\ 744]{Kreck2} or the end of \cite[p.~172]{Freedman_Quinn}). On the other hand, by Theorem \ref{thm_cw}, $H_{2} (M_{\infty}; \mathbb{Z})$ is a finitely generated abelian group. Thus $H_{2} (M_{\infty}; \mathbb{Z}) = 0$. By the Wang sequence
\[
H_2(M_{\infty};\mathbb Z) \to H_2(M_{\infty};\mathbb Z) \to H_2(M;\mathbb Z) \to 0
\]
of the fibration $M_{\infty} \rightarrow M \rightarrow S^1$, we also have $H_{2} (M; \mathbb{Z}) = 0$.

If $M$ is orientable, by \cite[Theorem 10.7A (2)]{Freedman_Quinn}, $M$ is homeomorphic to $S^{1} \times S^{3}$. Since $H_{2} (M_{\infty}; \mathbb{Z}) = 0$, we infer $M_{\infty}$ is spin. If $M$ is non-orientable, by \cite[Theorem 2 (2)]{Wang}, the Kirby-Siebenmann invariant of $M$ is $0$. Then by \cite[Theorem 1 (2)]{Wang}, $M$ is homeomorphic to $S^{1} \widetilde{\times} S^{3}$.
\end{proof}

Now we turn to the proof of Theorem \ref{thm_high_manifold}.

\begin{proof}[Proof of Theorem \ref{thm_high_manifold}]
(1). Let $M_{\infty}$ be the covering of $M$ with $\pi_{1} (M_{\infty}) = G \times 0$, then by Theorem \ref{thm_cw_homotopy}, $M_{\infty}$ is finitely dominated.  Since $\widetilde{K}_{0} (\mathbb{Z}[G]) =0$, the Wall's finiteness obstruction of $M_{\infty}$ vanishes (\cite[Theorem F]{Wall1}), therefore $M_{\infty}$ is homotopy equivalent to a finite CW complex. Now one can either apply Farrell's Fibration Theorem \cite[Theorem 6.4]{Farrell} (when $\dim M \ge 6$) or Cappell's Splitting Theorem to show that $M$ is a fiber bundle over $S^1$. To handle all cases uniformly, we stick to the second approach.

The idea is to interpret the fibering problem as a  splitting problem following \cite[footnote on p.~3]{Weinberger}.
Let $d \colon M_{\infty} \rightarrow M_{\infty}$ be the deck transformation corresponding $1 \in \mathbb{Z}$, $E= M_{\infty} \times I/ \sim$ be the mapping torus, where $I=[0,1]$ and $(x,0) \sim (d(x), 1)$. Then $E$ is a fiber bundle over $S^{1}$ with fiber $M_{\infty}$. There is a homotopy equivalence $\Phi: M \rightarrow E$ with $\Phi_{\sharp} (G \times 0) = \pi_{1} (M_{\infty})$. A homotopy inverse $E \rightarrow M$ is induced by the natural projection $M_{\infty} \times I \rightarrow M_{\infty} \rightarrow M$. In the fiber bundle $M_{\infty} \rightarrow E \rightarrow S^{1}$, all spaces are homotopy equivalent to finite CW complexes, and $E$ satisfies Poincar\'e duality, by \cite[Theorem 1]{Gottlieb}, we further infer $M_{\infty}$ is homotopy equivalent to a finite Poincar\'e complex.

By the fact $\mathrm{Wh} (G \times \mathbb{Z}) =0$ and Cappell's Splitting Theorem \cite[Corollary 6]{Cappell}, $\Phi$ splits along $M_{\infty} \subset E$. In other words, perturbing $\Phi$ homotopically, $\Phi^{-1} (M_{\infty}) = F$ is a CAT submanifold of $M$ and $\Phi: F \rightarrow M_{\infty}$ is a homotopy equivalence. Here an additional remark is needed for $\dim M =5$. We use the fact that every $1$-connected Poincar\'e space of formal dimension $4$ is homotopy equivalent to a closed TOP manifold. This follows from \cite[p.~161,~Theorem~10.1]{Freedman_Quinn} and \cite[p.~103,~Theorem~1.5]{MH}. (Note that \cite{MH} proved the result for manifolds. Its argument can be extended for Poincar\'e spaces thanks to \cite[Theorem~2.4]{Wall67}.)

Cutting $M$ along $F$, we get an $h$-cobordism $(M'; F', F'')$, where $M$ becomes $M'$, and $F$ becomes $F'$ and $F''$. Since $\pi_{1} (F) = G$ and $\mathrm{Wh} (G) =0$, $(M'; F', F'')$ is an $s$-cobordism. Thus $M'$ is isomorphic to $F \times I$. This shows that  $M$ is a fiber bundle over $S^1$ with fiber $F$.

(2). From (1), we see $M_{\infty}$ is isomorphic to $F \times \mathbb{R}$, and the deck transformation $d$ has the form
$d \colon  F \times \mathbb{R}  \rightarrow  F \times \mathbb{R}$, $
(x,t)  \mapsto  (f(x), t+1)$,
where $f$ is the monodromy of the fiber bundle. The homotopy equivalence $h: M \rightarrow M_{k}$ induces a homotopy equivalence $\overline{h}: M_{\infty} \rightarrow M_{\infty}$ such that $d^{k} \overline{h} = \overline{h} d^{\pm 1}$, where $d^{\pm 1} = d$ (resp.\ $d^{-1}$) if $h$ preserves (resp.\ reverses) the orientation of deck transformations. Let $i \colon F \hookrightarrow F \times \mathbb{R}$, $i (x) = (x,0)$ be the inclusion, $p_{1} \colon F \times \mathbb{R} \rightarrow F$ be the projection. Define $g = p_{1} \circ  \overline{h} \circ i \colon F \rightarrow F$. It's easy to see that $g$ satisfies the conclusion of (2).

(3). Both $M$ and $M_k$ are fiber bundles over $S^1$ with fiber $F$. Fix a fiber $F$ in each manifold. Since $h \colon M \to M_k$ is an isomorphism,  by Proposition \ref{prop_fiber} below, there is an automorphism $\phi$ of $M$ such that $\phi (F) = h^{-1} (F)$ and $\phi$ is pseudo-isotopic to the identity. Then $h \phi: M \rightarrow M_{k}$ is an isomorphism such that $h \phi (F) = F$. Since $\phi \simeq \mathrm{id}$, $h \phi$ preserves the orientation of deck transformation if and only if so does $h$. Cutting $M$ and $M_{k}$ along $F$, we obtain a CAT pseudo-isotopy $\psi: F \times I \rightarrow F \times I$ such that $f^{k} \psi_{0} = \psi_{1} f^{\pm 1}$, where $f^{\pm 1} = f$ (resp.\ $f^{-1}$) if $h$ preserves (resp.\ reverses) the orientation of deck transformations. Setting $g= \psi_{0}$, then $g$ is an automorphism of $F$ and
\[
f^{k} g = \psi_{1} \psi_{0}^{-1} g f^{\pm 1}.
\]
Clearly, $\psi_{1} \psi_{0}^{-1}$ is CAT pseudo-isotopic to $\mathrm{id}_{F}$. The proof is completed.
\end{proof}

\begin{remark}\label{rmk_fiber_extension}
In  the setting of Theorem \ref{thm_high_manifold}, if $G$ is more generally a finitely generated abelian group with $\widetilde{K}_{0} (\mathbb{Z}[G])=0$ and $\mathrm{Wh} (G)=0$, and $\dim M \geq 6$,  then one may obtain a ``virtual fibering theorem": there exists a $q_{0} >0$ such that $M_{q}$ is a fiber bundle over $S^1$ for all $q \geq q_{0}$. There do exist such $G$ other than $\mathbb{Z}^{r}$, for example $\mathbb{Z}/2$ and $\mathbb{Z}/3$. Here is the argument. By \cite[Corollary 6]{Cappell}, the splitting obstruction of $\Phi: M \rightarrow E$ is the image of $\tau (\Phi)$ in $\mathrm{Wh} (G \times \mathbb{Z}) / \mathrm{Wh} (G)$, where $\tau (\Phi)$ is Whitehead torsion of $\Phi$, and $(E, M_{\infty})$ is replaced homotopically with a finite CW pair. Let $[\tau (\Phi)]$ denote the image of $\tau (\Phi)$.
By the Bass-Heller-Swan decomposition (see the next section for a detailed discussion),
\[
\mathrm{Wh} (G \times \mathbb{Z}) / \mathrm{Wh} (G) = \widetilde{K}_{0} (\mathbb{Z}[G]) \oplus \mathrm{Nil} (\mathbb{Z}[G]) \oplus \mathrm{Nil} (\mathbb{Z}[G]) = \mathrm{Nil} (\mathbb{Z}[G]) \oplus \mathrm{Nil} (\mathbb{Z}[G]).
\]
The map $\Phi$ lifts to a homotopy equivalence $\Phi_{q}: M_{q} \rightarrow E_{q}$, and $[\tau (\Phi_{q})]$ is the transfer of $[\tau (\Phi)]$. By \cite[Theorem 6.11, step 3]{Farrell_Su}, $[\tau (\Phi_{q})]$ vanishes when $q$ is large enough. Thus $\Phi_{q}$ splits along $M_{\infty}$. The rest is the same as the proof of Theorem \ref{thm_high_manifold}. The properties (2) and (3) in Theorem \ref{thm_high_manifold} now also hold for $M_{q}$.
\end{remark}

\begin{proposition}\label{prop_fiber}
Let $N_{1}$ and $N_{2}$ be CAT bundles over $S^{1}$ with connected fibers, and let $F_{i} \subset N_{i}$ be a fiber. Suppose $\pi_{1} (N_{i}) = G \times \mathbb{Z}$ with $\pi_{1} (F_{i}) = G \times 0$, where $G$ is a free abelian group of finite rank. Suppose $h: N_{1} \rightarrow N_{2}$ is an isomorphism such that $h_{\sharp} (G \times 0) = G \times 0$, where $h_{\sharp}$ is the induced isomorphism $\pi_{1} (N_{1}) \rightarrow \pi_{1} (N_{2})$. If $\dim N_{i} \geq 6$, then there exists a CAT pseudo-isotopy $C: N_{1} \times I \rightarrow N_{1} \times I$ such that $C_{0} = \mathrm{id}_{N_{1}}$ and $C_{1} (F_{1}) = h^{-1} (F_{2})$, where $I= [0,1]$. If $\dim N_{i} = 5$, then the conclusion also holds for TOP.
\end{proposition}
\begin{proof}
Since $h$ is an isomorphism with $h_{\sharp} (G \times 0) = G \times 0$, following the construction in \cite[Lemma 4]{Wall65}, there exists a bicollared CAT submanifold  $W \subset N_{1} \times I$ which is an $h$-cobordism between  $F_{1} \times \{ 0 \}$ and $h^{-1} (F_{2}) \times \{ 1 \}$. (In \cite[Lemma 4]{Wall65}, $N_1$ is a product $F_1 \times S^1$, but the construction applies to general fiber bundles.)

Since $\pi_{1} (F_{1}) = G$ is a free abelian group, the Whitehead group $\mathrm{Wh} (G) =0$. The $s$-Cobordism Theorem implies that $W$ is isomorphic to $F_{1} \times I$. (When $\dim N_{1} =5$, this holds in the TOP category by Freedman's TOP $s$-Cobordism Theorem \cite[7.1A]{Freedman_Quinn}.) Cutting $N_{1} \times I$ along $W$, we get a connected manifold $U$ such that $\partial U$ consists of four parts $N_{1}'$, $N_{1}''$, $W'$ and $W''$. Here $W'$ and $W''$ are two copies of $W$, and $N_{1} \times \{ 0 \}$ (resp.\  $N_{1} \times \{ 1 \}$) becomes $N_{1}'$ (resp.\ $N_{1}''$).

We are going to show that  $(U; N_{1}', N_{1}'')$ is an s-cobordism relative to $W' \cup W''$. To achieve this, we first prove  the inclusion $\pi_{1} (N_{1}') \rightarrow \pi_{1} (U)$ is an isomorphism, which occupies the main part of the whole proof.

Note that $N_{1}'$ has two ends corresponding to $F_{1}$ in $N_{1} \times \{ 0 \}$. Let $F_{1}'$ denote the end in $W'$. Then both $\pi_{1} (F_{1}') \rightarrow \pi_{1} (N_{1}')$ and $\pi_{1} (F_{1}') \rightarrow \pi_{1} (W')$ are isomorphisms. By the commutative diagram
\[
\xymatrix{
G= \pi_{1} (F_{1}') \ar[d] \ar[r] & \pi_{1} (N_{1} \times \{0\}) \ar[d]^{\cong} \\
\pi_{1} (U) \ar[r] & \pi_{1} (N_{1} \times I) = G \times \mathbb{Z},
}
\]
we see $\pi_1(F_1') \to \pi_1(U)$ is injective. Therefore $\pi_1(N_1') \to \pi_1(U)$ is injective.

In the following, we use $pt$ to denote various chosen base points. Choose a base point $pt$ in $F_{1} \times \{ 0 \} \subset N_{1} \times \{ 0 \}$ and a loop $\gamma$ representing $(0,1) \in G \times \mathbb{Z} = \pi_{1} (N_{1} \times \{ 0 \}, pt)$. This $\gamma$ corresponds to a curve $\gamma'$ in $N_{1}'$. Let $\lambda: W' \rightarrow W''$ be the identification homeomorphism. Then $\gamma'$ connects $pt$ and $\lambda (pt)$. (Here $W'$ is identified with $W$.) We may assume $\gamma'$ is from $pt$ to $\lambda (pt)$. For each $a \in \pi_{1} (W', pt)$, define $\theta (a) = [\gamma'] \lambda_{*} (a) [\gamma']^{-1} \in \pi_{1} (U, pt)$. By the Second Van-Kampen Theorem \cite[p.\ 203]{Zastrow},
\[
G \times \mathbb{Z} = \pi_{1} (N_{1} \times I, pt) = \pi_{1} (U, pt) * \mathbb{Z} / \sim
\]
with
\begin{eqnarray*}
\sim & = & \{ a \cdot t^{-1} \cdot \theta (a)^{-1} \cdot t \mid a \in \mathrm{Im} (\pi_{1} (W', pt) \rightarrow \pi_{1} (U, pt)), t=1 \in \mathbb{Z} \} \\
& = & \{ a \cdot t^{-1} \cdot \theta (a)^{-1} \cdot t \mid a \in \mathrm{Im} (\pi_{1} (F_{1}', pt) \rightarrow \pi_{1} (U, pt)), t=1 \in \mathbb{Z} \}.
\end{eqnarray*}
For each $a \in \mathrm{Im} (\pi_{1} (F_{1}', pt) \rightarrow \pi_{1} (U, pt))$, we may consider $a$ as an element in $\pi_{1} (F_{1}', pt)$ since we have proved $\pi_{1} (F_{1}', pt) \rightarrow \pi_{1} (U, pt)$ is injective. Then $\theta (a)$ can be considered as an element in $\pi_{1} (N_{1}', pt)$. Clearly, $a = \theta (a)$ in $\pi_{1} (N_{1}', pt)$. Thus $a = \theta (a)$ in $\pi_{1} (U, pt)$. We further obtain
\begin{eqnarray*}
\sim & = & \{ a t^{-1} a^{-1} t \mid a \in \mathrm{Im} (\pi_{1} (F_{1}', pt) \rightarrow \pi_{1} (U, pt)), t=1 \in \mathbb{Z} \} \\
& = & \{ a t^{-1} a^{-1} t \mid a \in \mathrm{Im} (\pi_{1} (N_{1}', pt) \rightarrow \pi_{1} (U, pt)), t=1 \in \mathbb{Z} \}.
\end{eqnarray*}
Therefore, the following diagram commutes.
\[
\xymatrix{
\pi_{1} (N_{1}) \ar[d]_{\cong} & \ar[l]_-{\cong} \pi_{1} (N_{1}') * \mathbb {Z}/\sim \ar[r]^-{\cong} \ar[d] & \pi_{1} (N_{1}') \times \mathbb {Z} \ar[d] \\
\pi_{1} (N_{1} \times I) & \ar[l]_-{\cong} \pi_{1} (U) * \mathbb{Z}/\sim \ar[r] & \pi_{1} (U) \times \mathbb{Z}
}.
\]
Here the map $\pi_{1} (U) * \mathbb{Z}/\sim \rightarrow \pi_{1} (U) \times \mathbb{Z}$ is a quotient map and hence is surjective. Thus we infer the last vertical map of the above diagram is also surjective, which implies the surjectivity of $\pi_{1} (N_{1}') \rightarrow \pi_{1}(U)$. Up to now, we see $\pi_{1} (N_{1}') \rightarrow \pi_{1}(U)$ is an isomorphism.

Next we show $H_*(U,N_1'; \mathcal B)=0$ for any local system $\mathcal B$ on $U$. Since $\pi_{1} (U)$ is a direct summand of $\pi_{1} (N_{1} \times I)$, every local system on $U$ can be induced from one on $N_{1} \times I$. We have
\[
H_{*} (U, N_{1}'; \mathcal{B}) \cong H_{*} (U, N_{1}' \cup W' \cup W''; \mathcal{B}) \cong H_{*} (N_{1} \times I, N_{1} \times \{0\} \cup W; \mathcal{B}) = 0,
\]
where the second isomorphism is by excision, and the last equality holds since
\[
N_{1} \times \{ 0 \} \hookrightarrow (N_{1} \times \{ 0 \}) \cup W \hookrightarrow N_{1} \times I
\]
are homotopy equivalences. Thus $(U; N_{1}', N_{1}'')$ is an $h$-cobordism. Since $\mathrm{Wh}(G)=0$, it is an $s$-cobordism.

By the relative $s$-Cobordism Theorem, we can extend the isomorphism $F_{1} \times I \rightarrow W$ to an isomorphism $C': N'_{1} \times I \rightarrow U$ which results in the desired pseudo-isotopy $C$.
\end{proof}

\begin{remark}
Proposition \ref{prop_fiber} is true in general for groups $G$ with $\mathrm{Wh} (G)=0$, and in addition good in the sense of Freedman when $\dim N_{i} =5$.
\end{remark}

We conclude this section with the proofs of the corollaries on periodicity.

\begin{proof}[Proof of Corollary \ref{cor_period}]
Let $\varphi \colon \pi_1(M) = G \times \mathbb Z \to \mathbb Z$ be the projection. Then the assumptions of  Theorem \ref{thm_cw_period} are fulfilled. By the proof of Theorem \ref{thm_high_manifold} (2), the following diagram commutes up to homotopy.
\[
\xymatrix{
F \ar[d]_{\simeq} \ar[r]^{f} & F \ar[d]^{\simeq} \\
M_{\infty} \ar[r]^{d} & M_{\infty}
}
\]
Since the above vertical maps are homotopy equivalences, by Theorem \ref{thm_cw_period} (2), there exist $m_{j} >0$, prime to $k$, and $l_{j} >0$ such that $f_{*}^{m_{j}}$ (resp. $f_{*}^{l_{j}}$) equals the identity on $H_{j} (F; \mathbb{Z}) / T_{j}$ (resp. $H_{j} (F; \mathbb{Z})$) for $0 \leq j \leq \dim F$, where $T_{j}$ is the torsion of $H_{j} (F; \mathbb{Z})$. We finish the proof by defining $m$ (resp. $l$) as the product of these $m_{j}$ (resp. $l_{j}$).
\end{proof}

\begin{proof}[Proof of Corollary \ref{cor_5_top}]
In this case the fiber $F$ is a $1$-connected closed $4$-manifold. So $H_{\bullet} (F; \mathbb{Z})$ is a free abelian group. By Corollary \ref{cor_period}, $f_{*}^{m} = 1$ on $H_{\bullet} (F; \mathbb{Z})$ for some positive $m$ prime to $k$. This implies $f^m$ is topologically isotopic to $\mathrm{id}_F$  \cite[p.~161, Theorem~10.1]{Freedman_Quinn}.
\end{proof}

For a smooth manifold $F$, the mapping class group $\mathcal M(F)$ is the group of path components of the group of orientation preserving automorphisms of $F$, and the Torelli group $\mathcal{T} (F)$ is the subgroup of $\mathcal{M} (F)$ consisting of elements which act trivially on $H_{\bullet} (F; \mathbb{Z})$.

\begin{lemma}\label{lem_torelli}
Let $F$ be an $(n-1)$-connected closed smooth almost parallelizable $2n$-manifold with $n\ge 3$. Suppose $a \in \mathcal{M} (F)$, $b \in \mathcal{T} (F)$, and $b^{k} = aba^{-1}$ for some $k>1$. Then $b$ is a torsion element of $\mathcal{T} (F)$.
\end{lemma}
\begin{proof}
The mapping class group of such an $F$ is computed in \cite{Kreck1}. By \cite[Theorem 2]{Kreck1}, there is an exact sequence
\begin{equation}\label{lem_torelli_1}
0 \rightarrow C \rightarrow \mathcal{T} (F) \rightarrow A \rightarrow 0,
\end{equation}
where $A$ is a finitely generated abelian group, $C$ is a finite subgroup of $\mathcal{T} (F)$, and $C$ is contained in the center of $\mathcal{M} (F)$. Let $T$ be the torsion subgroup of $A$, and $[b]$ be the image of $b$ under the projection $\mathcal{T} (F) \rightarrow A/T$.

We show $[b] =0$ by contradiction.  If $[b] \ne 0$, then $A/T \cong \mathbb{Z}^{u}$ for some $u>0$ and
\[
[b] = (w_{1}, \dots, w_{u}) \neq 0
\]
via this isomorphism. These $w_{i}$ have their greatest common divisor $d>0$, which we denote by $\mathrm{gcd} ([b]) = d$. Then $\mathrm{gcd} ([b]^{k}) = kd$. Since $C$ is contained in the center of $\mathcal{M} (F)$, the automorphism $\mathcal T(F) \to \mathcal T(F)$, $x \mapsto a x a^{-1}$ induces  an automorphism of $A/T$. By the assumption $b^{k} = aba^{-1}$, this automorphism maps $[b]$ to $[b]^{k}$. So $\mathrm{gcd} ([b]^{k}) = \mathrm{gcd} ([b]) = d$. We infer $d=kd$, which contradicts the fact $d>0$ and $k>1$.

As $[b]=0$, we see $b^{s} \in C$ for some $s>0$. Since $C$ is finite, the conclusion follows.
\end{proof}

\begin{proof}[Proof of Corollary \ref{cor_monodromy}]
Since $\pi_{1} (M) = \mathbb{Z}$, we see $F$ is $1$-connected. By the statement (3) in Theorem \ref{thm_high_manifold}, replacing $g$ with $g^{2}$ if necessary, we may assume that $f^{k}$ is pseudo-isotopic to $gfg^{-1}$ for some $k>1$ and $g$ preservers the orientation of $F$. We also know $\dim F \geq 6$. By Cerf's Theorem \cite{Cerf2}, $f^{k}$ is smooth isotopic to $gfg^{-1}$. Furthermore, by Corollary \ref{cor_period}, $f^{l_{1}}$ acts on $H_{\bullet} (F; \mathbb{Z})$ trivially for some $l_{1} >0$. Thus $f^{l_{1}}$ represents an element $[f^{l_{1}}] \in \mathcal{T} (F)$, $g$ represents an element $[g] \in \mathcal{M} (F)$, and $[f^{l_{1}}]^{k} = [g] [f^{l_{1}}] [g]^{-1}$.

Since $\pi_{i} (M) =0$ for all $1<i<n$, we infer $F$ is $(n-1)$-connected. Since $F$ is almost parallelizable and $n \geq 3$, we finish the proof by taking $a$ and $b$ in Lemma \ref{lem_torelli} as $[g]$ and $[f^{l_{1}}]$ respectively.
\end{proof}

\section{Manifold: Non-Fibering Examples}\label{sec_non-fiber}
In this section, we demonstrate various constructions of self-covering manifolds which are not fiber bundles over $S^1$. We first construct the $5$-dimensional example as stated in Theorem \ref{thm_5_diff}.

\begin{proof}[Proof of Theorem \ref{thm_5_diff}]
The construction of such a manifold follows from the theory of $4$-manifolds (see e.g. \cite[p.1]{Weinberger}). Let $E_8$ be the positive definite unimodular symmetric quadratic form over $\mathbb Z$ of rank $8$. Let $V$ be the $1$-connected closed topological 4-manifold with intersection form $E_{8} \oplus E_{8}$. The Kirby-Siebenmann invariant $\mathrm{ks} (V) = 0$ (c.~f.~\cite[Chapter 10]{Freedman_Quinn}). Let $N = V\times S^{1}$. The obstruction to the existence of a smooth structure on $N$ is the Kirby-Siebenmann class $\mathrm{ks} (N) \in H^{4}(N; \mathbb{Z}_{2})$ (\cite[p.~318, Theorem 5.4]{Kirby_Siebenmann}). Fix a point $x \in S^1$, the inclusion $i \colon V \rightarrow V \times \{ x \} \hookrightarrow N$ induces an isomorphism $H^{4} (N; \mathbb{Z}_{2}) \to H^4(V;\mathbb Z_2)$, with $i^{*} \mathrm{ks} (N) = \mathrm{ks} (V) =0$. Therefore $N$ is smoothable.

Equip $N$ with a smooth structure. For each integer $u \geq 0$, the cyclic covering $N_{2^{u}}$ is also a smooth manifold and homeomorphic to $N$. There are only finitely many distinct smooth structures on the topological manifold $N$ (\cite[p.\ 318, Theorem\ 5.4]{Kirby_Siebenmann}). Thus there exist $u_{1} < u_{2}$ such that $N_{2^{u_{1}}}$ is diffeomorphic to $N_{2^{u_{2}}}$. Let $M= N_{2^{u_{1}}}$ and $k= 2^{u_{2} - u_{1}}$. Then $M$ is diffeomorphic to $M_{k}$ and $k>1$.

It remains to show that $M_{q}$ is not a DIFF (PL) bundle over $S^{1}$ for all $q$. In fact, if $M_{q}$ were such a bundle, then its fiber $F$ would be a smooth $1$-connected closed manifold. (We may assume $F$ is connected by taking a finite cover of $S^{1}$.) Since $F$ is homotopy equivalent to $V$, its intersection form is also $E_{8} \oplus E_{8}$ which is positive definite. This contradicts  Donaldson's theorem that a positive definite intersection form of a smooth $4$-manifold must be diagonal (\cite[Theorem A]{Donaldson}).
\end{proof}

The main contents of this section are the proofs of Theorems \ref{thm_no_finite} and \ref{thm_no_fiber}, in which the obstructions to fibration come from the algebraic $K$-groups of the fundamental group. We first recall some preliminaries of algebraic $K$-theory, both for setting up the notations and for the convenience of the reader.

Let $\Gamma $ be a group, and let $\mathbb Z [\Gamma] = \{ \sum_{i} n_{i} g_{i} \ | \ n_{i} \in \mathbb{Z}, \ g_{i} \in  \Gamma \}$ be the integral group ring, with conjugation $\sum_{i} n_{i} g_{i} \mapsto \sum_{i} n_{i} g_{i}^{-1}$. The group $K_0(\mathbb Z [\Gamma])$ is the Grothendieck group of finitely generated projective $\mathbb Z [\Gamma]$-modules, i.~e., elements in $K_{0} (\mathbb{Z}[\Gamma])$ are of the form $[P]-[Q]$, where $P$ and $Q$ are finitely generated projective left $\mathbb{Z}[\Gamma]$-modules.
There is a canonical involution (or conjugation) on $K_{0} (\mathbb{Z}[\Gamma])$, defined as follows. For a finitely generated projective left $\mathbb Z [\Gamma]$-module $P$, its dual module $\mathrm{Hom}_{\mathbb{Z}[\Gamma]} (P, \mathbb{Z}[\Gamma])$ is a right $\mathbb{Z}[\Gamma]$-module. We convert this right module to a left module via the conjugation of $\mathbb{Z}[\Gamma]$ and denote it by $P^{*}$. Equivalently, if $P$ is represented by an idempotent matrix $A$ over $\mathbb{Z}[\Gamma]$, then $P^{*}$ is represented by the conjugate transpose of $A$. The map $[P] \mapsto [P^{*}]$ extends to an involution of $K_{0} (\mathbb{Z}[\Gamma])$, denoted by $y \mapsto y^{*}$. A caveat: for an abelian group $\Gamma$, the group ring $\mathbb{Z}[\Gamma]$ is a commutative ring, and hence the dual module $\mathrm{Hom}_{\mathbb{Z}[\Gamma]} (P, \mathbb{Z}[\Gamma])$ is automatically a left module, but this is \emph{not} the conjugate $P^{*}$ in the  definition.

Let $\Gamma = \mathbb{Z}/p$ be a cyclic group of prime order $p$. A nice reference for the basic facts about $K_{0} (\mathbb{Z}[\mathbb{Z}/p])$ is \cite[$\S$1-3]{Milnor71}. Let $\xi_{p} = e^{\frac{2 \pi i}{p}}$ be a $p$-th root of unity. There is an evident epimorphism $\mathbb{Z} [\mathbb{Z}/p] \rightarrow \mathbb{Z}[\xi_{p}]$, which, by Rim's Theorem, induces an isomorphism between the reduced $K$-groups
\[
\widetilde{K}_{0} (\mathbb{Z}[\mathbb{Z}/p]) \overset{\cong}{\longrightarrow} \widetilde{K}_{0} (\mathbb{Z}[\xi_{p}]).
\]
The group $\widetilde{K}_{0} (\mathbb{Z}[\xi_{p}])$ is isomorphic to the ideal class group of $\mathbb{Z}[\xi_{p}]$, whose elements are represented by the fractional ideals $L$ of $\mathbb{Z}[\xi_{p}]$ (c.~f.~\cite[p.\ 96]{Atiyah_Macdonald}). The inverse of $[L]$ is represented by
\[
L^{-1} = (\mathbb{Z}[\xi_{p}] : L) = \{ x \in \mathbb{Q}(\xi_{p}) \mid x \cdot L \subseteq \mathbb{Z}[\xi_{p}] \}.
\]
By \cite[Theorem 11.6~(c)]{Eisenbud}, we have the following lemma.
\begin{lemma}\label{lem_ideal_inverse}
$\mathrm{Hom}_{\mathbb{Z}[\xi_{p}]} (L, \mathbb{Z}[\xi_{p}]) = L^{-1}$.
\end{lemma}

The conjugation of $\mathbb{Z}[\mathbb{Z}/p]$ corresponds to the complex conjugation of $\mathbb{Z}[\xi_{p}]$. By Lemma~\ref{lem_ideal_inverse}, for $[L] \in \widetilde{K}_{0} (\mathbb{Z}[\xi_{p}])$, we obtain a formula for the involution
\begin{equation}\label{eqn_k0_conjugate}
[L]^{*} = \left[ \overline{L^{-1}}^{\mathbb{C}} \right] = - \left[ \overline{L}^{\mathbb{C}} \right],
\end{equation}
where $\overline{L}^{\mathbb{C}}$ is the complex conjugation of $L$.

There are classical results on the $(\pm 1)$-eigenspace of $\widetilde{K}_{0} (\mathbb{Z}[\xi_{p}])$ under the complex conjugation. More precisely, $\mathbb{Z}[\xi_{p}]$ contains the maximal real subring $\mathbb{Z}[\xi_{p} + \xi_{p}^{-1}]$, and $\widetilde{K}_{0} (\mathbb{Z}[\xi_{p} + \xi_{p}^{-1}])$ is also isomorphic to its ideal class group (\cite[Corollary 1.11]{Milnor71}). There are exact sequences (see \cite[Theorem 4..2, p.~82; Theorem 4.4, p.~84]{Lang})
\[
0 \rightarrow \widetilde{K}_{0} (\mathbb{Z} [\xi_{p} + \xi_{p}^{-1}]) \overset{\iota}{\rightarrow} \widetilde{K}_{0} (\mathbb{Z}[\xi_{p}]),
\]
and
\[
0 \rightarrow \ker \mathrm{N} \rightarrow \widetilde{K}_{0} (\mathbb{Z}[\xi_{p}]) \overset{\mathrm{N}}{\rightarrow} \widetilde{K}_{0} (\mathbb{Z} [\xi_{p} + \xi_{p}^{-1}]) \rightarrow 0,
\]
where $\iota$ is induced by the inclusion and $\mathrm{N}$ is the norm map. The order of $\widetilde{K}_{0} (\mathbb{Z} [\xi_{p} + \xi_{p}^{-1}])$ (resp. $\ker \mathrm{N}$) is called the second (resp. first) factor of the class number of $\mathbb Z[\xi_p]$ and denoted by $h_{p}^{+}$ (resp. $h_{p}^{-}$). Clearly, the image of $\iota$ has order $h_{p}^{+}$ and is contained in the $(+1)$-eigenspace of complex conjugation. Furthermore, $\ker \mathrm{N}$ is exactly the $(-1)$-eigenspace of complex conjugation (see \cite[Theorem 4.4, p.~84]{Lang}).

By (\ref{eqn_k0_conjugate}) and the above argument, we obtain the following lemmas.
\begin{lemma}\label{lem_class_number}
If $h_{p}^{-}$ (resp. $h_{p}^{+}$) has a prime factor $q$, then there exists $y \in \widetilde{K}_{0} (\mathbb{Z}[\mathbb{Z}/p])$ such that the order of $y$ is $q$, and $y^{*} = y$ (resp. $y^{*} = -y$).
\end{lemma}

\begin{lemma}\label{lem_k0}
If both $h_{p}^{-}$ and $h_{p}^{+}$ have odd prime factors, then $\widetilde{K}_{0} (\mathbb{Z}[\mathbb{Z}/p])$ contains elements $y_{1}$ and $y_{2}$ such that the orders of $y_{1}$ and $y_{2}$ are odd primes, $y_{1}^{*} = y_{1}$ and $y_{2}^{*} = -y_{2}$.
\end{lemma}

\begin{remark}\label{rmk_class_number}
There do exist $p$ such that both $h_{p}^{-}$ and $h_{p}^{+}$ have odd prime factors. By comparing the table in \cite[p.~352]{Washington} and the main table in \cite[p.\ 935]{Schoof}, we can find a few such $p$, for example, $p=191$ etc.
\end{remark}

Next we briefly discuss the Whitehead group $\mathrm{Wh}(\Gamma)$. Recall that elements in $\mathrm{Wh}(\Gamma)$ are represented by matrices in the general linear group $\mathrm{GL} (\mathbb{Z} [\Gamma])$. Let $\Gamma= \mathbb{Z}/p \times \mathbb{Z}$. Then $
\mathbb{Z} [\mathbb{Z}/p \times \mathbb{Z}] = \mathbb{Z} [\mathbb{Z}/p] [t,t^{-1}]
$
is the Laurent polynomial ring over $\mathbb{Z} [\mathbb{Z}/p]$. The Bass-Heller-Swan decomposition (c.f. \cite[Chapter 6]{Farrell_Su}) gives an isomorphism
\begin{equation}\label{eqn_BHS}
\mathrm{Wh} (\mathbb{Z}/p \times \mathbb{Z}) \cong \mathrm{Wh} (\mathbb{Z}/p) \oplus \widetilde{K}_{0} (\mathbb{Z}[\mathbb{Z}/p]) \oplus \mathrm{Nil} (\mathbb{Z}[\mathbb{Z}/p]) \oplus \mathrm{Nil} (\mathbb{Z}[\mathbb{Z}/p]),
\end{equation}
where the summand $\mathrm{Wh} (\mathbb{Z}/p)$ is induced from the inclusion $\mathbb{Z}/p \times 0 \hookrightarrow \mathbb{Z}/p \times \mathbb{Z}$, and the summand $\widetilde{K}_{0} (\mathbb{Z}[\mathbb{Z}/p])$ means there is a monomorphism
\[
\phi \colon \ \widetilde{K}_{0} (\mathbb{Z}[\mathbb{Z}/p]) \rightarrow \mathrm{Wh} (\mathbb{Z}/p \times \mathbb{Z})
\]
defined as follows. Each element in $\widetilde{K}_{0} (\mathbb{Z}[\mathbb{Z}/p])$ is of the form $[P]-[Q]$, where $P$ and $Q$ are $n \times n$ idempotent matrices over $\mathbb{Z}[\mathbb{Z}/p])$ for some $n$. Then
\[
\phi ([P]-[Q]) = [tP+ I_{n} - P] - [tQ+ I_{n} - Q],
\]
where $I_{n}$ is the $n \times n$ identity matrix.

There is also an involution on $\mathrm{Wh}(\Gamma)$: for $x \in \mathrm{Wh}(\Gamma)$ represented by an $A \in \mathrm{GL} (\mathbb{Z} [\Gamma])$, the involution $x^{*}$ is represented by the conjugate transpose $A^{*}$ of $A$.
\begin{lemma}\label{lem_k1_conjugate}
For each $y \in \widetilde{K}_{0} (\mathbb{Z}[\mathbb{Z}/p])$, $\phi(y)^{*} = - \phi \left( y^{*} \right)$.
\end{lemma}
\begin{proof}
Suppose $y = [P] - [Q]$. We have
\[
\phi ([P])^{*} = [tP + (I-P)]^{*} = [t^{-1} P^{*}+ I - P^{*}] = -[t P^{*}+ I - P^{*}] = - \phi \left( [P]^{*} \right).
\]
Here the third equality is because $(t^{-1} P^{*}+ I - P^{*}) (t P^{*}+ I - P^{*}) = I$. Similarly, $\phi ([Q])^{*} = - \phi \left( [Q]^{*} \right)$. The conclusion follows.
\end{proof}

Given any integer $k>0$, for the finite index subgroup $\mathbb{Z}/p \times (k \mathbb{Z}) \subseteq \mathbb{Z}/p \times \mathbb{Z}$, there is a transfer homomorphism
\[
\mathrm{tr} \colon \ \mathrm{Wh} (\mathbb{Z}/p \times \mathbb{Z}) \rightarrow \mathrm{Wh} (\mathbb{Z}/p \times (k \mathbb{Z})).
\]
There is also a canonical isomorphism $\mathbb{Z}/p \times \mathbb{Z} \rightarrow \mathbb{Z}/p \times (k \mathbb{Z})$ sending $1 \in \mathbb{Z}$ to $k \in k \mathbb{Z}$.

\begin{lemma}\label{lem_transfer}
Let $x$ be in the summand $\widetilde{K}_{0} (\mathbb{Z}[\mathbb{Z}/p])$ of (\ref{eqn_BHS}). Then $\mathrm{tr} (x) =x$ via the canonical isomorphism $\mathbb{Z}/p \times \mathbb{Z} \cong \mathbb{Z}/p \times (k \mathbb{Z})$ for each $k>0$.
\end{lemma}
\begin{proof}
By the Bass Projection (\cite[Chapter 6]{Farrell_Su}), the projection of $\mathrm{tr} (x)$ on the second summand of (\ref{eqn_BHS}) is $x$, the projections to the Nil-groups are $0$. Since $\mathrm{Wh} (\mathbb{Z}/p)$ is free (\cite[Theorem 5.6]{Oliver}) and $x$ is of finite order (\cite[p.~98]{Atiyah_Macdonald}), the projection of $\mathrm{tr} (x)$ on the first summand of (\ref{eqn_BHS}) is also $0$. Thus $\mathrm{tr} (x) =x$.
\end{proof}

Now we are ready to prove Theorem \ref{thm_no_fiber}.

\begin{proof}[Proof of Theorem \ref{thm_no_fiber}]
Since both $h_{p}^{-}$ and $h_{p}^{+}$ have odd prime factors, by Lemmas \ref{lem_k0} and \ref{lem_k1_conjugate}, there exists an $x$ in the summand $\widetilde{K}_{0} (\mathbb{Z}[\mathbb{Z}/p])$ of (\ref{eqn_BHS}) such that $2x \neq 0$ and $(-1)^{n} x^{*} = -x$.

Let $N=F \times S^{1}$. Since $n \geq 5$, by the classification of $h$-cobordisms (\cite[Theorem 11.1]{Milnor66}), there is a smooth $h$-cobordism $(W; N, M)$ such that the Whitehead torsion $\tau (W,N) = x$. Now $M$ satisfies the statement (1).

Note that $F$ is automatically orientable since $\pi_{1} (F) = \mathbb{Z}/p$ for $p>2$, and hence $W$ is orientable. By the Duality Theorem of Whitehead torsion (\cite[p.~394]{Milnor66}), we have
\[
\tau (W,M) = (-1)^{n} \tau (W,N)^{*} = (-1)^{n} x^{*} = -x.
\]
For each integer $k>1$, consider the $k$-fold covering $(W_{k}; N_{k}, M_{k})$ of $(W; N, M)$ associated to $\mathbb{Z}/p \times (k \mathbb{Z}) \subset \mathbb{Z}/p \times \mathbb{Z}$. It satisfies that $\tau (W_{k}, N_{k}) = \mathrm{tr} (x)$. By Lemma \ref{lem_transfer}, we infer $\tau (W_{k}, N_{k}) = x$ via the canonical isomorphism $\mathbb{Z}/p \times \mathbb{Z} \cong \mathbb{Z}/p \times (k \mathbb{Z})$.

Identify $N$ with $N_{k}$. We consider $(W_{k}; N_{k}, M_{k})$ as $(W_{k}; N, M_{k})$ and hence
\[
\tau (W_{k}, N) = \tau (W_{k}, N_{k}) = \tau (W, N) =x.
\]
By the classification of $h$-coborisms (\cite[Theorem~11.3]{Milnor66}), $M$ is diffeomorphic to $M_{k}$. We obtain the statement (2).

Finally, let's conclude statement (3). Let $i: M \hookrightarrow W$ be the inclusions and $r: W \rightarrow N$ be the retraction. Then the Whitehead torsion of the homotopy equivalence $\varphi = r \circ i$ is
\[
\tau (\varphi) = \tau (i) + \tau (r) = \tau (W, M) - \tau (W, N) = -2x \neq 0.
\]
For each $q \geq 1$, the homotopy equivalence $\varphi$ lifts to a homotopy equivalence $\varphi_{q}: M_{q} \rightarrow N_{q}$ with Whitehead torsion (by Lemma \ref{lem_transfer})
\begin{equation}\label{thm_no_fiber_1}
\tau (\varphi_{q}) = \tau (\varphi) \neq 0.
\end{equation}

Suppose $\Pi: M_{q} \rightarrow S^{1}$ is a topological fiber bundle. Taking a finite cover of $S^{1}$ and changing the orientation of $S^{1}$ if necessary, we may assume
\[
\Pi_{\sharp}: \ \ \pi_{1} (M_{q}) = \mathbb{Z}/p \times \mathbb{Z} \rightarrow \pi_{1} (S^{1}) = \mathbb{Z}
\]
is the coordinate projection. Let $p_{2}: N_{q} = F \times S^{1} \rightarrow S^{1}$ be the projection, then $p_{2} \circ \varphi_{q}$ is homotopic to $ \Pi$. By the homotopy lifting property, $\varphi_{q}$ is homotopic to a fiber homotopy equivalence $\psi_{q}: M_{q} \rightarrow N_{q}$. By \cite[Theorem A]{Anderson}, we have
\[
\tau (\psi_{q}) = z \cdot \chi (S^{1})=0,
\]
where $z$ lies in the image of $\mathrm{Wh} (\pi_{1} (F)) \rightarrow \mathrm{Wh} (\pi_{1} (N_{q}))$, and $\chi (S^{1})$ is the Euler number of $S^{1}$. This contradicts (\ref{thm_no_fiber_1}) since $\tau(\psi_q)=\tau(\varphi_q)$. Here we note that, though \cite[Theorem~A]{Anderson} was originally proved in PL, it certainly can be extend to TOP by \cite{Chapman}, since topological manifolds are Hilbert cube manifold factors.
\end{proof}

We conclude this section with the proof of  Theorem \ref{thm_no_finite}. The construction of the manifolds $M$ in Theorem  \ref{thm_no_finite} consists of several steps. First of all, by Theorem \ref{thm_cw_homotopy}, the infinite cyclic cover $M_{\infty}$ of  $M$ is finitely dominated. We start by constructing a finitely dominated Poincar\'e complex.

\begin{proposition}\label{prop_poincare}
Let $G$ be a finitely presented group and fix an element $\theta \in \widetilde{K}_{0} (\mathbb{Z}[G])$. For each $n \geq 4$, there exists a finitely dominated Poincar\'{e} complex $Y$ of formal dimension $n$ such that:
\begin{enumerate}
\item $\pi_{1} (Y) = G$;

\item the Wall's finiteness obstruction of $Y$ is $\sigma (Y) = \theta + (-1)^{n} \theta^{*} \in \widetilde{K}_{0} (\mathbb{Z}[G])$;

\item the Spivak normal fibration of $Y$ is trivial.
\end{enumerate}
\end{proposition}
\begin{proof}
The existence of $Y$ satisfying (1) and (2) follows from  \cite[Theorem 1.5]{Wall67}. To achieve (3), one needs a slight modification of the construction there.

Let $N'$ be a smooth closed stably parallelizable $n$-manifold with $\pi_{1} (N') = G$. One way to construct such $N'$ is by surgery: build a connected sum of some copies of $S^{n-1} \times S^{1}$, and then get the desired fundamental group by framed $1$-surgeries. By \cite[Theorem 2]{Milnor61}, the resulting manifold is stably parallelizable.

Suppose $n=2k$ or $2k+1$. Let $P$ be a projective $\mathbb{Z}[G]$-module such that $[P]=(-1)^{k} \theta \in \widetilde K_{0}(\mathbb Z G)$ modulo free modules. Let $F$ be a free module such that $F = P \oplus Q$. Suppose $F$ has rank $r$. Let $N$ be the connected sum of $N'$ with $r$ copies of $S^{k} \times S^{k}$ or $S^{k} \times S^{k+1}$. By \cite{Milnor61} again, $N$ is stably parallelizable. Attaching to $N$ (usually infinitely many) cells of dimension $k$ and $k+1$, and also $k+2$ if $n=2k+1$, with appropriate attaching maps, by the argument of \cite{Wall67}, we obtain a complex $Y$ satisfying Poincar\'{e} duality and $\sigma (Y) = \theta + (-1)^{n} \theta^{*}$.

It remains to show that the Spivak normal fibration of $Y$ is trivial. Obviously, the inclusion $i: N \hookrightarrow Y$ is a degree $1$ map between Poincar\'{e} spaces. Let
\[
T(N) = N \times D^{m} / N \times S^{m-1} \qquad \text{and} \qquad T(Y) = Y \times D^{m} / Y \times S^{m-1}
\]
be the Thom spaces of the trivial spherical fibration over $N$ and $Y$ respectively, where $m>2$. The inclusion $i$ induces another inclusion $T(i): T(N) \hookrightarrow T(Y)$, which maps a fundamental class of $T(N)$ to that of $T(Y)$. Since $N$ is stably parallelizable, $N \times S^{m-1}$ is the Spivak normal fibration of $N$ for $m$ large enough. In other words, there exists a map $h: S^{n+m} \rightarrow T(N)$ which maps a fundamental class of $S^{n+m}$ to that of $T(N)$. Thus $T(i) \circ h$ maps a fundamental class of $S^{n+m}$ to that of $T(Y)$. This implies that $Y \times S^{m-1}$ is the Spivak normal fibration of $Y$, which is trivial.
\end{proof}

Let $p$ be the prime number in Theorem \ref{thm_no_finite}. By Lemma \ref{lem_k0} and Proposition \ref{prop_poincare}, we obtain the following lemma.

\begin{lemma}\label{lem_poincare}
There exists a finitely dominated Poincar\'{e} space $Y$ of formal dimension $4$ and $5$ respectively such that:
\begin{enumerate}
\item $\pi_{1} (Y) = \mathbb{Z}/p$;

\item the Wall's finiteness obstruction $\sigma (Y)$ is of odd prime order;

\item the Spivak normal fibration of $Y$ is trivial.
\end{enumerate}
\end{lemma}

The space $Y$ will be the one in Lemma \ref{lem_poincare} for the rest of this section.

\begin{lemma}\label{lem_smooth_Poincare}
There exists a closed DIFF manifold $M'$ with dimension $7$ and $8$ respectively such that $M'$ is stably parallelizable and homotopy equivalent to $Y \times S^{2} \times S^{1}$.
\end{lemma}
\begin{proof}
We know the Spivak normal fibration of $Y$ is trivial, hence that of $Y \times S^{1}$ is also trivial. Choose the trivial vector bundle to represent the Spivak normal fibration of $Y \times S^{1}$. Since $Y$ is finitely dominated, $Y \times S^{1}$ is homotopy equivalent to a finite Poincar\'{e} complex. Consider the surgery problem with target $Y \times S^{1}$ endowed with the trivial vector bundle, i.e., the problem of turning a degree $1$ normal map from a smooth manifold to $Y \times S^{1}$ to a homotopy equivalence. The surgery obstruction might be non-trivial. However, by the product formula for surgery obstruction \cite[Theorem IV.1.1]{Morgan}, the surgery obstruction vanishes if one takes a product with $S^2$. Therefore there is a closed smooth  manifold $M'$ homotopy equivalent to $Y \times S^{1} \times S^{2}$. Furthermore, $M'$ is stably parallelizable, since we take the trivial vector bundle over $Y \times S^1 \times S^2$ as the reduction of the Spivak normal fibration.
\end{proof}

Clearly, $\pi_{1} (M') = \mathbb{Z}/p \times \mathbb{Z}$ and the cyclic covering $M'_{q}$ of $M'$ with $\pi_{1} (M'_{q}) = \mathbb{Z}/p \times (q \mathbb{Z})$ is homotopy equivalent to $Y \times S^{2} \times S^{1}$ for all $q \geq 1$. There is a natural homotopy equivalence $h'_{1,q} \colon M'_{1} \rightarrow M'_{q}$ which induces an isomorphism $\mathbb Z/p \times \mathbb Z \to \mathbb Z/p \times q\mathbb Z$, $(a,b) \mapsto (a,qb)$ between fundamental groups. Let $M'_{\infty}$ denote the infinite cyclic cover of $M'$ with $\pi_{1} (M'_{\infty}) = \mathbb{Z}/p$.

\begin{lemma}
The manifold $M'$ can be modified such that it satisfies an additional property: the Whitehead torsion $\tau (h'_{1,2}) \in \mathrm{Wh} (\mathbb{Z}/p) \subseteq \mathrm{Wh} (\mathbb{Z}/p \times \mathbb{Z})$.
\end{lemma}
\begin{proof}
Assume $\tau (h'_{1,2}) = (x_{1}, x_{2}, x_{3}, x_{4})$ according to the decomposition (\ref{eqn_BHS}). Consider the homotopy equivalence between covering spaces $h'_{q,2q} \colon M'_{q} \rightarrow M'_{2q}$ induced by $h'_{1,2}$. Then $\tau (h'_{q,2q})$ is the transfer of $\tau (h'_{1,2})$. By \cite[Theorem 6.11, step 3]{Farrell_Su}, the transfers $\mathrm{tr} (x_{3}) = \mathrm{tr} (x_{4}) =0$ when $q$ is large enough. Clearly, $\mathrm{tr} (x_{1}) = q x_{1}$, and by Lemma \ref{lem_transfer}, $\mathrm{tr} (x_{2}) = x_{2}$. Thus replacing $M'$ with a finite cover if necessary, we may assume $x_{3} = x_{4} =0$.

Now we show that indeed $x_{2} =0$, which will complete the proof. The homotopy equivalence $h'_{1,2}$ induces a homotopy equivalence $\overline{h'}_{1,2} \colon M'_{\infty} \rightarrow M'_{\infty}$. Since $\overline{h'}_{1,2}$ induces the identity map on the fundamental group, by \cite[Corollary 4.8]{Siebenmann}, we obtain
\[
\sigma (M'_{\infty}, \varepsilon_{+}) = \sigma (M'_{\infty}, \overline{h'}_{1,2}^{-1} (\varepsilon_{+})) + x_{2} \in \widetilde{K}_{0} (\mathbb{Z}[\mathbb{Z}/p]).
\]
Here $\sigma$ is Siebenmann's finiteness obstruction of ends, $\varepsilon_{+}$ is the positive end of $M'_{\infty}$ corresponding to $(0,1) \in \mathbb{Z}/p \times \mathbb{Z} = \pi_{1} (M')$, and $\overline{h'}_{1,2}^{-1} (\varepsilon_{+})$ is the end induced from $\varepsilon_{+}$ by $\overline{h'}_{1,2}$. Since $\overline{h'}_{1,2}$ preserves the orientation of deck transformation, we infer $\overline{h'}_{1,2}^{-1} (\varepsilon_{+}) = \varepsilon_{+}$, and hence
\[
\sigma (M'_{\infty}, \varepsilon_{+}) = \sigma (M'_{\infty}, \varepsilon_{+}) + x_{2},
\]
which implies $x_{2} =0$.
\end{proof}

\begin{lemma}\label{lem_7_manifold}
There exists a closed DIFF $7$-manifold $M^{7}$ satisfying the following:
\begin{enumerate}
\item $M^{7}$ is stably parallelizable;

\item $M^{7}$ is homotopy equivalent to $Y \times S^{2} \times S^{1}$;

\item there exists a simple homotopy equivalence $h_{1,2}: M^{7} \rightarrow M^{7}_{2}$ which induces the isomorphism $\mathbb{Z}/p \times \mathbb{Z} \rightarrow \mathbb{Z}/p \times (2\mathbb{Z})$, $(a,b) \mapsto (a,2b)$ between fundamental groups.
\end{enumerate}
\end{lemma}
\begin{proof}
Consider the $7$-manifold $M'$ constructed in Lemma \ref{lem_smooth_Poincare}. Replacing $M'$ with a double cover, we may assume $\tau (h'_{1,2}) =2y$ for some $y \in \mathrm{Wh}(\mathbb{Z}/p)$. Let $(W; M', M^{7})$ be a smooth $h$-cobordism $(W; M', M^{7})$ with $\tau (W,M') =-y$. Clearly, $M^{7}$ satisfies (1) and (2).

The group $\mathrm{Wh}(\mathbb{Z}/p)$ is free (\cite[p.~3]{Oliver}) and the involution $x \mapsto x^{*}$ on it is the identity (\cite[Corollary 7.5]{Oliver}). Hence by the Duality Theorem (\cite[p.~394]{Milnor66}),
$$\tau (W, M^{7}) = (-1)^{7} (-y)^{*} =y.$$
The inclusions $M' \hookrightarrow W$ and $M^{7} \hookrightarrow W$ result in a homotopy equivalence $\phi: M^{7} \rightarrow M'$ with $\tau (\phi) = \tau (W, M^{7}) - \tau (W,M') = 2y$.
Taking the double covering, we get $\phi_{2}: M^{7}_{2} \rightarrow M'_{2}$ and $\tau (\phi_{2}) = 2 \tau(\phi) = 4y$. Define
\[
h_{1,2}= \phi_{2}^{-1} \circ h'_{1,2} \circ \phi: \ \ M^{7} \rightarrow M^{7}_{2},
\]
where $\phi_{2}^{-1}$ is a homotopy inverse of $\phi_{2}$. Then $h_{1,2}$ is a homotopy equivalence with torsion
\[
\tau (h_{1,2}) = \tau (\phi) + \tau (h'_{1,2}) - \tau (\phi_{2}) = 0.
\]
The induced isomorphism between the fundamental groups is clearly ${h_{1,2}}_{\sharp} (a,b) = (a,2b)$. This finishes the proof.
\end{proof}

\begin{remark}
When $\dim M' =8$, the argument in Lemma \ref{lem_7_manifold} doesn't work because $\tau (\phi) =0$ in this case.
\end{remark}

\begin{lemma}\label{lem_8_mainfold}
There exist an $8$-manifold $M^8$ and a simple homotopy equivalence $h_{1,m}: M^{8} \rightarrow M^{8}_{m}$ for some $m>1$. Both $M^{8}$ and $h_{1,m}$ satisfy the corresponding properties as $M^{7}$ and $h_{1,2}$ do in Lemma \ref{lem_7_manifold}.
\end{lemma}
\begin{proof}
Consider the $8$-manifold $M'$ constructed in Lemma \ref{lem_smooth_Poincare}. The homotopy equivalence $h'_{1,2}: M' \rightarrow M'_{2}$ induces a family of homotopy equivalences $h'_{2^{q}, 2^{q+1}}: M'_{2^{q}} \rightarrow M'_{2^{q+1}}$ for all $q \geq 0$. Composing these maps, we obtain homotopy equivalences $h'_{2^{r}, 2^{s}}: M'_{2^{r}} \rightarrow M'_{2^{s}}$ for all $s > r \geq 0$. Clearly, $\tau (h'_{2^{r}, 2^{s}}) \in \mathrm{Wh} (\mathbb{Z}/p)$. Let $h'_{2^{s}, 2^{r}}$ be a homotopy inverse of $h'_{2^{r}, 2^{s}}$.

Let $\nu_{q}$ be the normal bundle of $M'_{2^{q}}$. Consider the family of normal maps
\[
(M'_{2^{q}}, \nu_{q}) \rightarrow (M', (h'_{1,2^{q}})^{*} \nu_{q}).
\]
They represent elements in the set of normal structures $\mathscr N^{\mathrm{DIFF}}(M')$, which is the set of bordism classes of normal maps to $M'$ (\cite[Proposition 9.43]{Ranicki2}). The set $\mathscr N^{\mathrm{DIFF}}(M')$ is identified with $[M', G/O]$. The groups $\pi_{*} (G/O)$ and $\pi_{*} (BO)$ are finitely generated abelian groups, and $\pi_{*} (G/O) \otimes \mathbb{Q} \rightarrow \pi_{*} (BO) \otimes \mathbb{Q}$ are isomorphisms (\cite[\S.~9.2]{Ranicki2}). Applying the Puppe sequence to a CW structure of $M'$, we infer $[M', G/O]$ and $[M', BO]$ are finitely generated abelian groups, and
\[
[M', G/O] \otimes \mathbb{Q} \rightarrow [M', BO] \otimes \mathbb{Q}
\]
is an isomorphism. Since $M'_{2^{q}}$ are stably parallelizable for all $q$, the images of the above family of normal maps under the map $[M', G/O] \rightarrow [M', BO]$
are all trivial.
So they only represent finitely many elements in $\mathscr N^{\mathrm{DIFF}}(M')$.

Thus there exist some $s>r$ and a normal map over the bordisms
\[
\Psi^{9}: \ \ (W^{9}; M'_{2^{r}}, M'_{2^{s}}) \rightarrow (M' \times [0,1]; M' \times \{ 0 \}, M' \times \{ 1 \}),
\]
where the bundle over $M' \times [0,1]$ is trivial. Replacing $M'$ with $M'_{2^{r}}$, we obtain a normal map over the bordisms
\[
\Psi^{9}: \ \ (W^{9}; M', M'_{m}) \rightarrow (M' \times [0,1]; M' \times \{ 0 \}, M' \times \{ 1 \}),
\]
where $m= 2^{s-r}$, $\Psi^{9}|_{M'} = \mathrm{id}$ and $\Psi^{9}|_{M'_{m}} = h'_{m,1}$.

Let's apply an argument similar to the one in the proof of Lemma \ref{lem_smooth_Poincare}. By the product formula for surgery obstruction \cite[Theorem IV.1.1]{Morgan}, we obtain a homotopy equivalence of bordisms
\[
\Psi^{11}: \ \ (W^{11}; M' \times S^{2}, M'_{m} \times S^{2})  \rightarrow (M' \times [0,1]; M' \times \{ 0 \}, M'\ \times \{ 1 \}) \times S^{2}
\]
such that $\Psi^{11}|_{M' \times S^{2}} = \mathrm{id}$ and $\Psi^{11}|_{M'_{m} \times S^{2}} = h'_{m,1} \times \mathrm{id}_{S^{2}}$. The $h$-cobordism $(W^{11}; M' \times S^{2}, M'_{m} \times S^{2})$ induces a homotopy equivalence $\phi: M' \times S^{2} \rightarrow M'_{m} \times S^{2}$ such that
\begin{eqnarray*}
\tau (\phi) & = & \tau (W^{11}, M' \times S^{2}) - \tau (W^{11}, M'_{m} \times S^{2}) \\
& = & \tau (W^{11}, M' \times S^{2}) - (-1)^{10} \tau (W^{11}, M' \times S^{2})^{*}.
\end{eqnarray*}
Since the involution on $\mathrm{Wh}(\mathbb{Z}/p)$ is the identity, by (\ref{eqn_BHS}), we see
\begin{equation}\label{lem_8_mainfold_1}
\tau (\phi) = (0, x_{2}, x_{3}, x_{4}).
\end{equation}
Here we use a fact: the involution maps the first and second factors of (\ref{eqn_BHS}) into themselves, and switches the third and fourth factors. On the other hand, by the following commutative diagram of homotopy equivalences
$$
\xymatrix{
M'_{m} \times S^{2} \ar[d] \ar[r]^-{h'_{m,1} \times \mathrm{id}_{S^{2}}} & M'_{1} \times S^{2} \times \{1\} \ar[d] \\
W^{11} \ar[r] & M' \times S^{2} \times [0,1] \\
M' \times S^{2} \ar[u] \ar[r]^-{\mathrm{id}} & M' \times S^{2} \times \{0\} \ar[u]
},
$$
we see $\phi \simeq h'_{1,m} \times \mathrm{id}_{S^{2}}$. By the product formula for Whitehead torsion \cite[Corollary~1.3]{Kwun_Szczarba},
\begin{equation}\label{lem_8_mainfold_2}
\tau (\phi) = \tau (h'_{1,m}) \cdot \chi (S^{2}) = 2 \tau (h'_{1,m}) \in \mathrm{Wh}(\mathbb{Z}/p).
\end{equation}
Since $\mathrm{Wh}(\mathbb{Z}/p)$ is free, comparing (\ref{lem_8_mainfold_1}) and (\ref{lem_8_mainfold_2}), we see $\tau (h'_{1,m}) =0$. Setting $M^{8} = M'$ and $h_{1,m} = h'_{1,m}$, we finish the proof.
\end{proof}

\begin{proof}[Proof of Theorem \ref{thm_no_finite}]
Firstly, we prove the theorem in the case of $\dim M$ is $9$ or $10$.

Consider the $*$-dimensional manifold $M^{*}$ in Lemmas \ref{lem_7_manifold} and \ref{lem_8_mainfold}, where $*$ is $7$ or $8$. We know that there is a simple homotopy equivalence $h_{1,m}: M^{*} \rightarrow M^{*}_{m}$ for some $m>1$. This induces a family of simple homotopy equivalences $h_{m^{r}, m^{s}}: M^{*}_{m^{r}} \rightarrow M^{*}_{m^{s}}$ for all $s>r \geq 0$. By an argument as that in the proof of Lemma \ref{lem_8_mainfold}, replacing $M^{*}$ with a finite cover if necessary, we obtain a normal map over bordisms
\[
\Psi^{*+1}: \ \ (W^{*+1}; M^{*}, M^{*}_{k}) \rightarrow (M^{*} \times [0,1]; M^{*} \times \{0\}, M^{*} \times \{1\}),
\]
where $k>1$, $\Psi^{*+1}|_{M^{*}} = \mathrm{id}$ and $\Psi^{*+1}|_{M^{*}_{k}}$ is a simple homotopy equivalence. Now similar to the proof of Lemma \ref{lem_8_mainfold}, doing surgery for simple homotopy equivalence, we obtain a smooth $s$-cobordism $(W^{*+3}; M^{*} \times S^{2}, M^{*}_{k} \times S^{2})$. Here we apply the product formula for surgery obstruction again. Note that the formula is also valid in the case of simple homotopy equivalence (\cite[Theorem IV.1.1]{Morgan}).

Let $M^{*+2} = M^{*} \times S^{2}$. Then $M^{*+2}$ is of dimension $9$ or $10$ and is diffeomorphic to $M^{*+2}_{k} = M^{*}_{k} \times S^{2}$. Furthermore, $M^{*+2}$ is stably parallelizable because so is $M^{*}$. On the other hand,
\[
M^{*+2} \simeq Y \times S^{2} \times S^{2} \times S^{1},
\]
and hence the infinite cyclic cover
\[
M^{*+2}_{\infty} \simeq Y \times S^{2} \times S^{2}.
\]
By Gersten's product formula (\cite[Theorem 0.1]{Gersten}), the Wall's finiteness obstruction is
\[
\sigma (M^{*+2}_{\infty}) = \sigma (Y) \cdot \chi(S^{2} \times S^{2}) = 4 \sigma (Y).
\]
By Lemma \ref{lem_poincare}, the order of $\sigma (Y)$ is odd prime. Then $\sigma (M^{*+2}_{\infty})$ is nonzero, i.e. $M^{*+2}_{\infty}$ is not homotopy equivalent to any finite CW complex. In summary, $M^{*+2}$ is a desired manifold $M$ in dimensions $9$ and $10$.

By taking the product $M^{*+2} \times S^{2q}$, we obtain desired manifolds $M$ in all higher dimensions.
\end{proof}

\begin{remark}\label{rmk_class_dimension}
By Lemmas \ref{lem_class_number} and \ref{lem_k1_conjugate} and the proof of Theorem \ref{thm_no_fiber}, to guarantee the conclusion of Theorem \ref{thm_no_fiber}, it suffices to require $h_{p}^{+}$ (resp. $h_{p}^{-}$) has an odd prime factor when $n$ is odd (resp. even). In contrast, by Lemma \ref{lem_class_number} and Proposition \ref{prop_poincare}, to guarantee the conclusion of Theorem \ref{thm_no_finite}, it suffices to require $h_{p}^{-}$ (resp. $h_{p}^{+}$) has an odd prime factor when $n$ is odd (resp. even).
\end{remark}

\section{Further Questions}\label{sec_question}
We conclude this paper with some questions.

This paper only deals with the fibering problem with abelian fundamental groups. A generalization to arbitrary fundamental groups can be formulated as follows. Let $M$ be a closed manifold with fundamental groups $G \rtimes \mathbb{Z}$, where $G$ is finitely presented (and not necessarily abelian). Let $M_{k}$ be the $k$-cover of $M$ with fundamental group $G \rtimes (k\mathbb{Z})$. Suppose there is a homotopy equivalence $h: M \rightarrow M_{k}$ such that $h_{\sharp} (G) = G$ for some $k>1$, where $h_{\sharp}: \pi_{1} (M) \rightarrow \pi_{1} (M_{k})$ is induced by $h$. Does $M$ always fiber over $S^{1}$? To answer this question, we need to know if $M_{\infty}$ is homotopy equivalent to a finite CW complex, where $M_{\infty}$ is the infinite cyclic cover of $M$ with $\pi_{1} (M_{\infty}) = G$. This leads to the following question about generalizations of Theorem \ref{thm_cw_homotopy}.

\begin{question}
Suppose $X$ is a CW complex with $\pi_{1} (X) = G \rtimes \mathbb{Z}$, where $G$ is finitely presented. Let $X_{k}$ and $X_{\infty}$ be covers of $X$ with $\pi_{1} (X_{k}) = G \rtimes (k\mathbb{Z})$ and $\pi_{1} (X_{\infty}) = G$. Suppose $h: X \rightarrow X_{k}$ is a homotopy equivalence for some $k>1$ such that $h_{\sharp} (G) = G$. If $X$ is homotopy equivalent to a CW complex of finite type (resp. finitely dominated), then does this imply $X_{\infty}$ is as well?
\end{question}

Because our arguments heavily rely on the theory of commutative algebra, they do not apply if $\pi_{1} (X)$ is not abelian.
\begin{question}
Is there a proof of Theorem \ref{thm_cw_homotopy} without using extensively commutative algebra?
\end{question}

We have answered Question \ref{que_pi1_z} almost completely with one exception: $n=4$ and CAT is DIFF (=PL). As pointed out in the introduction, to answer this missing case is equivalent to answer the following question, which is open to the best of the authors' knowledge.
\begin{question}\label{que_4_diff}
Do $S^{1} \times S^{3}$ and $S^{1} \tilde{\times} S^{3}$ have a unique smooth structure?
\end{question}

Finally, we are interested in the following question.
\begin{question}
Is there an $n$-manifold $M$ satisfying the conclusion of Theorem \ref{thm_no_finite} with $n<9$, especially $n \geq 5$?
\end{question}


\end{document}